\documentclass[11pt,reqno]{amsart}
\usepackage[utf8]{inputenc}
\usepackage{amsmath}
\usepackage{amsfonts}
\usepackage{amssymb}
\usepackage{amsthm}
\usepackage{mathtools}
\usepackage{dsfont}
\usepackage{xcolor}
\usepackage{tikz}
\usepackage{enumitem}
\usepackage{relsize}
\usepackage{graphicx}
\usepackage{wrapfig}
\usetikzlibrary{patterns}
\usepackage{float}
\usepackage[left=3cm,right=3cm,top=3cm,bottom=4cm]{geometry}
\usepackage[normalem]{ulem}
\usepackage{comment}
\title[Stability of discrete-time Hawkes process with inhibition]{Stability of discrete-time Hawkes process with inhibition: towards a general condition}
\date{}

\author[M. Costa]{Manon Costa}
\address{Manon Costa\\ Institut de Math\'ematiques de Toulouse, CNRS UMR 5219 \\
	Universit\'e Toulouse III Paul Sabatier
	\\ 118 route
	de Narbonne, F-31062 Toulouse cedex 09.} \email{manon.costa@math.univ-toulouse.fr}

\author[P. Maillard]{Pascal Maillard}
\address{Pascal Maillard\\Institut Universitaire de France and Institut de Math\'ematiques de Toulouse, CNRS UMR 5219 \\
	Universit\'e Toulouse III Paul Sabatier
	\\ 118 route
	de Narbonne, F-31062 Toulouse cedex 09.} \email{pascal.maillard@math.univ-toulouse.fr}

\author[A. Muraro]{Anthony Muraro}
\address{Anthony Muraro\\ Institut de Math\'ematiques de Toulouse, CNRS UMR 5219 \\
	Universit\'e Toulouse III Paul Sabatier
	\\ 118 route
	de Narbonne, F-31062 Toulouse cedex 09.} \email{anthony.muraro@math.univ-toulouse.fr}

\numberwithin{equation}{section}
\newtheorem{thm}{Theorem}[section]
\newtheorem*{thm*}{Theorem}

\newtheorem{lemma}[thm]{Lemma}
\newtheorem{prop}[thm]{Proposition}
\newtheorem{conjecture}[thm]{Conjecture}
\theoremstyle{definition}
\newtheorem{definition}[thm]{Definition}

\theoremstyle{remark}

\newtheorem*{conjecture_rappel_b<1}{Conjecture \ref{conjecture_b<1}}

\newcommand{\R}{\mathbb{R}}

\newcommand{\N}{\mathbb{N}}

\renewcommand{\P}{\mathbb{P}}
\newcommand{\E}{\mathbb{E}}
\newcommand{\Disc}{\text{Disc}}

\newcommand*{\proba}[1]{\mathbb{P}\left( #1 \right)}

\begin{document}
\maketitle

\begin{center}
	\textsc{Universit\'e de Toulouse}
	\smallskip
\end{center}
\begin{abstract}
In this paper, we study a discrete-time analogue of a Hawkes process, modelled as a Poisson autoregressive process whose parameters depend on the past of the trajectory. The model is characterized to allow these parameters to take negative values, modelling inhibitory dynamics. More precisely, the model is the stochastic process $(\Tilde X_n)_{n\ge0}$ with parameters $a_1,\ldots,a_p \in \R$, $p\in\N$ and $\lambda > 0$, such that for all $n\ge p$, conditioned on $\Tilde X_0,\ldots,\Tilde X_{n-1}$, $\Tilde X_n$ is Poisson distributed with parameter 
\[
\left(a_1 \Tilde X_{n-1} + \cdots + a_p \Tilde X_{n-p} + \lambda \right)_+.
\]
This process can be seen as a discrete time Hawkes process with inhibition with a memory of length $p$. 
We first provide a sufficient condition for stability in the general case which is the analog of a condition for continuous time Hawkes processes from \cite{costa_renewal_2020}. 
We then focus on the case $p=3$, extending the results derived for the $p=2$ case in a previous work \cite{Costa_Maillard_Muraro_2024}. In particular, we show that the process may be stable even if one of the coefficients $a_i$ is much greater than one.
\end{abstract}
\bigskip

\section{Introduction}

Hawkes processes are a class of point processes initially introduced by Hawkes  \cite{hawkes_spectra_1971,hawkes_cluster_1974}. These processes became well-known within the scientific community for their applications in modelling seismic events, specifically earthquakes. The fundamental aim in the formulation of this stochastic process was to formalize the concept of \textit{excitation} over time, which describes the phenomenon that the occurrence of an event increases the probability of another event, and so on. In a Hawkes process, the excitation's intensity depends on the entire history of the process, and influences its behaviour over time. An essential characteristic of the Hawkes process is its non-Markovian nature (in most cases), primarily because it depends on the complete past trajectory of the process.

In recent years, there has been a rise of interest in Hawkes processes (see for example \cite{bremaud_stability_1996,kirchner_perspectives_2017}), particularly in generalized versions that incorporate the notion of \textit{inhibition} \cite{costa_renewal_2020,raad_stability_2020,cattiaux_limit_2021}. Inhibition denotes the possibility of existence of events that decrease the probability of future events to occur. Mathematically, adding inhibition into the intensity causes new mathematical challenges. In particular obtaining necessary and sufficient conditions for the existence of stationary version remains an open question. To address this issue, we proposed in a previous paper \cite{Costa_Maillard_Muraro_2024} an analogous, discrete, and simplified process. This approach was devised to yield insights into the simplified process that could, in turn, offer intuition for potential generalizations or interpretations concerning the continuous Hawkes process. 

Our model is the following, for $p \in \N^*$ and initial condition $(\Tilde X_0, \dots, \Tilde X_{-p+1})$, the discrete-time Hawkes process with memory of length $p$  satisfies that for all $n\ge1,$
\begin{equation}
\label{defDiscreteHawkesP}
 \Tilde X_n \sim \mathcal{P}\left( \left( a_1\Tilde X_{n-1} + \cdots + a_p\Tilde X_{n-p} + \lambda \right)_+ \right)\qquad \text{conditionally on $X_{n-p}, \cdots, X_{n-1}$,}
\end{equation}
where $\mathcal{P(\rho)}$ denotes the Poisson distribution with parameter $\rho$. The parameters here $a_1,\ldots,a_p$ are real numbers, and $\lambda > 0$. The notation $(\cdot)_+$ refers to the ReLU function defined on $\R$ by :
$$(x)_+ := \max(0,x).$$ 

The process $\tilde{X}_n$ defined in \eqref{defDiscreteHawkesP} is called \textit{Poisson-autoregressive} processes, and has been deeply studied in the "linear case" that is, the same definition with every parameters being positive, so that the positive part $(\cdot)_+$ vanishes in \eqref{defDiscreteHawkesP} (for more details see \cite{ferland_integer-valued_2006,fokianos_interventions_2010}). Kirchner \cite{kirchner_perspectives_2017} proved the weak convergence of an infinite memory version ($p=+\infty$) to the continuous-time Hawkes process in the linear case. 
As far as we know, there are very few papers dealing with a non-linear version of such a process (see for example \cite{fokianos_log-linear_2011}). In \cite{Costa_Maillard_Muraro_2024}, we gave an almost complete characterization of stability of this process (excluding boundary cases) in the particular case $p=2$. 

In this paper, we  first investigate the  case of general $p$. We show that a sufficient condition for stability is that $(a_1)_+ + \cdots + (a_p)_+ < 1$. This result is analogous to a recent result for continuous time Hawkes processes from \cite{costa_renewal_2020}. We then focus on the case $p=3$. We provide a sufficient condition for stability of the process which can hold true even if one of the coefficients $a_i$ is much greater than one. This condition involves the existence of non-real eigenvalues (of arbitrary modulus) of the polynomial $$X^3 - a_1 X^{2} - a_{2}X - a_3.$$
This is to be compared with the classical condition for stability of solutions to the \textit{linear} recurrence equation $u_n = a_1 u_{n-1} + a_2 u_{n-2} + a_3 u_{n-3}$, which imposes that all eigenvalues of this polynomial be of modulus smaller than one. In particular, we show the existence of a wide range of parameters for which the discrete-time Hawkes process with inhibition is stable, but solutions to the linear recurrence equation grow exponentially.

We complement these results with two sets of sufficient conditions for instability of the discrete-time Hawkes process with inhibition as well as some numerical experiments in other ranges of parameters. Together, these results illustrate the complex effects of the non-linearity induced by the ReLU function and the negativity of the parameters.

\section{Main results}

\subsection{The general setting}
\label{subsec:setting}
Let us start by considering the general case, of a Poisson auto-regressive process with memory $p$ defined as in \eqref{defDiscreteHawkesP} :
for $p \in \N^*$ and initial condition $(\Tilde X_0, \dots, \Tilde X_{-p+1})$, the discrete-time Hawkes process with memory of length $p$ satisfies that conditionally on $X_{n-p}, \cdots, X_{n-1}$
\begin{equation*}
 \Tilde X_n \sim \mathcal{P}\left( \left( a_1\Tilde X_{n-1} + \cdots + a_p\Tilde X_{n-p} + \lambda \right)_+ \right)\qquad \forall n \geq 1,
\end{equation*}

The naturally associated Markov chain $(X_n)_{n \geq 0} \in \N^p$ writes 
\begin{equation}
\label{MarkovChainp}
\forall n \geq 0, ~  X_n := \left(\Tilde X_{n}, \Tilde X_{n-1}, \ldots,\Tilde X_{n-p+1}\right)  \in \N^p.
\end{equation}
More formally, for an initial condition $x = (x_1,\ldots,x_p)\in\N^p$ and parameters $a_1,\ldots,a_p \in \R,\lambda>0$, knowing $X_n = x$ the next step $X_{n+1}$ of the Markov chain will be :
$$X_{n+1} = (\ell, x_1,\ldots,x_{p-1}),$$
where $\ell$ is a realization of a Poisson random variable of parameter 
$$s_{x_1,\ldots,x_p}:= (a_1x_1+\cdots + a_px_p + \lambda)_+ \in [0,+\infty).$$
The transition matrix of the Markov chain $(X_n)_{n\ge0}$ is thus given for $(x_1,\ldots,x_p)\in\N^p$ and $\ell\in\N$ by:
\begin{equation}
\label{transitionP}
P\left( (x_1,\ldots,x_p), (\ell,x_1,\ldots,x_{p-1}) \right) =  \dfrac{e^{-s_{x_1,\ldots,x_p}}s_{x_1,\ldots,x_p}^\ell}{\ell !},
\end{equation}
and $P(x,y)= 0$ if $y$ is not of the above form.
In particular, the Markov chain is (weakly) irreducible in the sense of \cite{dmps_markov_chains}, since $(0,\ldots,0)$ is an accessible atom from any starting state (see Section~\ref{sec:Lyap} below). We say that it is \emph{geometrically ergodic,} if it admits a stationary distribution $\pi$ and if the law of $X_n$ converges exponentially fast in total variation to $\pi$.

\subsubsection*{Existing results for the linear case}

The case where all parameters $a_1,\ldots, a_p$ are nonnegative is well understood. The process $(\Tilde X_n)$ is an \textit{Integer-Valued GARCH process} (INGARCH($0,p$) process, see \cite{ferland_integer-valued_2006} and \cite{liu_systematic_2023}), also known in the literature as an \textit{Auto-regressive Conditional Poisson process} (ACP($p$) see \cite{heinen_modelling_2003}) and its long time behaviour depends on $\sum_{i=1}^p a_i$.
Namely, 
\begin{itemize}
\item If $\displaystyle \sum_{i=1}^p a_i < 1$, the process admits a stationary version. In fact, the Markov chain $(X_n)_{n\ge0}$ is geometrically ergodic. This result can be well understood when considering the branching interpretation of the INGARCH processes \cite{kirchner_perspectives_2017}, in which $\sum_{i=1}^p a_i$ is the mean number of offspring generated by a living individual.
\item If $\displaystyle \sum_{i=1}^p a_i > 1$, then $\Tilde X_n$ grows exponentially in $n$ almost surely. For completeness we will provide a simple proof of this general result, postponed to the Appendix \ref{app:Trans_p} at the end of this paper. 
\end{itemize}

\subsubsection*{A first result for the non-linear version}
The following result partially extends the first result mentioned above to the case of possibly non-negative parameters $a_1,\ldots,a_p$. It is a discrete analogue of a result in  \cite{costa_renewal_2020} for continuous-time Hawkes processes. 
\begin{thm}
If $(a_1)_+ + \dots + (a_p)_+ < 1$, then $(X_n)_{n\ge0}$ is a geometrically ergodic Markov chain.
\label{thm:recurrencePHDp}
\end{thm}
As announced in the previous section, Theorem \ref{thm:recurrencePHDp}  naturally extends the known condition for non-negative parameters. However, this sufficient condition is quite restrictive because it only considers the positive part of the coefficients and therefore neglects the effect of inhibition. For example, it is natural to inquire whether, in the case where one of the parameters is greater than 1, one can choose the other parameters sufficiently negatively to obtain stability. This question appears to be quite challenging to resolve, as the non-linearity comes into play in such instances. A generalization of the study of this process to an arbitrary number $p$ of parameters seems beyond reach at the moment. 
Therefore, we confine ourselves in the following to the specific case of $p=3$, the case $p=2$ having been addressed in \cite{Costa_Maillard_Muraro_2024}.

\subsection{Three-parameter discrete-time Hawkes process}
In the case where $p=3$, we choose for the sake of clarity to relabel the parameters as $a_1=a$, $a_2=b$ and $a_3=c$.
The discrete time process  $(\tilde{X}_n)_{n \geq 0}$ is now defined with initial condition $(\tilde{X}_0, \tilde{X}_{-1}, \tilde{X}_{-2})$ as $n\ge 1$:
\begin{equation}
\text{conditioned on $\tilde{X}_{n-3},\tilde{}_{n-2},\tilde{X}_{n-1}$: }\tilde{X}_n\sim \mathcal{P}\left( \left(a \tilde{X}_{n-1} + b \tilde{X}_{n-2} + c \tilde{X}_{n-3} + \lambda \right)_+ \right). 
\label{defPHD}
\end{equation}
The associated Markov chain writes $X_n \coloneqq (\tilde X_n,\tilde X_{n-1},\tilde X_{n-2})$ and takes values in $\N^3$.

Our objective is to complete the results of Theorem \ref{thm:recurrencePHDp} by considering cases where the three parameters have different signs and one of them is larger than 1.


\subsubsection{Existing results in the case $c=0$}
When $c=0$, the process $(\Tilde X_n)_{n \geq 0}$ is a discrete Hawkes process with memory of length two, which is studied in \cite{Costa_Maillard_Muraro_2024}. The main result of that article is the following :

\begin{thm*}[Theorem 1 in \cite{Costa_Maillard_Muraro_2024}] 
\label{theoRecall}
Assume $c=0$. Define the function 
\[
b_*(a) = \begin{cases}
1 & a \le 0\\
1-a & a\in(0,2)\\
-\frac{a^2}{4} & a\ge 2
\end{cases}
\]
and define the following sets :
\begin{align}
\label{eq:def_R}
\mathcal R &= \left\{(a,b)\in\R^2: b<b_*(a)\right\}\\
\label{eq:def_T}
\mathcal T &= \left\{(a,b)\in\R^2: b>b_*(a)\right\}.
\end{align}

\begin{itemize}
    \item If $(a,b) \in \mathcal{R}$, then the sequence $(\Tilde X_n)_{n\ge0}$ converges in law as $n\to\infty$ (in fact, the Markov chain $(X_n)_{n\ge0}$ is geometrically ergodic).
    \item If $(a,b) \in \mathcal{T}$, then the sequence $(\Tilde X_n)_{n\ge0}$ satisfies that
almost surely 
$$\Tilde X_n+\Tilde X_{n+1}\underset{n\to\infty}{\longrightarrow}+\infty\,.$$
\end{itemize}
\end{thm*}
Let us highlight that the parameters $a$ and $b$ have asymmetric roles : for any values of $a$, if $b$ is chosen small enough, then the Markov chain $(X_n)_{n\ge0}$ will be recurrent, while as soon as $b>1$, it will be transient.

The rest of the article is devoted to the case where $c\neq0$. We classified our results into two classes : sufficient conditions for recurrence and sufficient conditions for transience. Since we did not manage to classify entirely the space of parameters $(a,b,c)\in\R^3$ we complete our theoretical results with numerical results. 
\medskip

Before stating our results, let us introduce important notions that will be necessary in the proofs.
\subsubsection{Linear recurrence, polynomials and discriminant}

Due to the concentration property of the Poisson distribution, as long as the parameter of the Poisson random variable in the definition of $(\Tilde X_n)_{n \geq 0}$ remains positive, we expect that the process behaves similarly to the deterministic sequence $(\Tilde u_n)_{n \geq 0}$, which is defined as follows:
\begin{equation}
\label{linarRec}
\Tilde u_n = a \Tilde u_{n-1} + b \Tilde u_{n-2} + c \Tilde u_{n-3} + \lambda.
\end{equation}
It is a well-known fact that the roots of the polynomial
\begin{equation}
\label{PolynomialP}
P(X) := X^3 - aX^2 - bX - c,
\end{equation}
play a significant role in the identification of sequences that satisfy the recurrence \eqref{linarRec}. More precisely, the stability of the linear recurrence $(\Tilde u_n)_{n \geq 0}$ (in the sense that sequences satisfying recurrence \eqref{linarRec} remain bounded for all $n \in \N$) is characterized  by the condition :
\begin{equation}
\label{condStabLin}
\max \{ |\zeta|, ~ P(\zeta) = 0 \} < 1.
\end{equation}


The results we will state on the recurrence of the Markov chain $(X_n)_{n \geq 0}$ use another quantity associated with the polynomial $P$, namely its \textit{discriminant}. 
Concerning the polynomial $P$ we are interested in, its discriminant can be easily computed, yielding : 
\begin{equation}
    \label{eq:disc}
    \text{Disc}(P) := a^2b^2+4b^3-4a^3c-18abc-27c^2
\end{equation}
The sign of $\Disc(P)$ determines the number of roots of $P$ that are complex. For a degree 3 polynomial, on the one hand, if $\Disc(P) < 0$, the polynomial $P$ has an unique simple real root, and two complex conjugate roots. On the other hand, if $\Disc(P) \geq 0$, all the roots of $P$ are real numbers, and they are simple when $\Disc(P) > 0$. 
For more details about resultant and discriminant, we refer the reader to chapter 12 of \cite{gelfand_discriminants_1994}.

In Figure \ref{fig:disc}, we propose a graphical representation of the surface of $\R^3$ delimited by $\Disc (P) = 0$.
\begin{figure}[!h]
        \centering
        \includegraphics[scale=1]{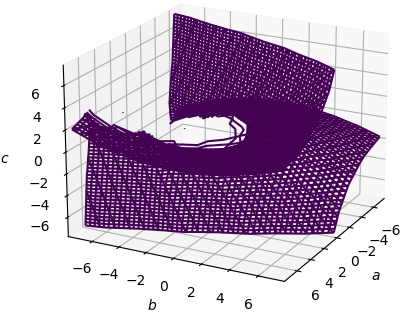}
    \caption{Graphical representation of the surface of $\R^3$ delimited by $\Disc(P) = 0$. The axis represents the different values of parameters $a,b,c \in \R$.
    Note that $\Disc(P)$ is negative outside the two surfaces. }
    \label{fig:disc}
\end{figure}
\begin{figure}[!h]
    \begin{minipage}[c]{0.46\linewidth}
        \centering
        \includegraphics[width=0.8\linewidth]{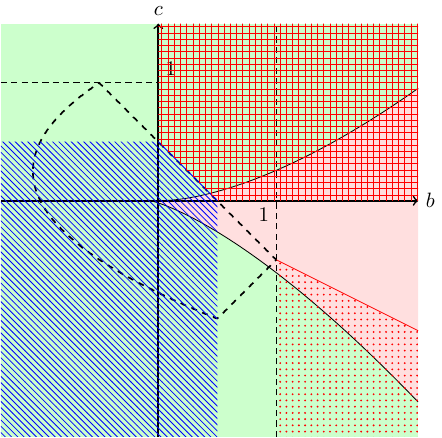}\\
        $a=0.5$
    \end{minipage}
    \hfill%
    \begin{minipage}[c]{0.46\linewidth}
        \centering
        \includegraphics[width=0.8\linewidth]{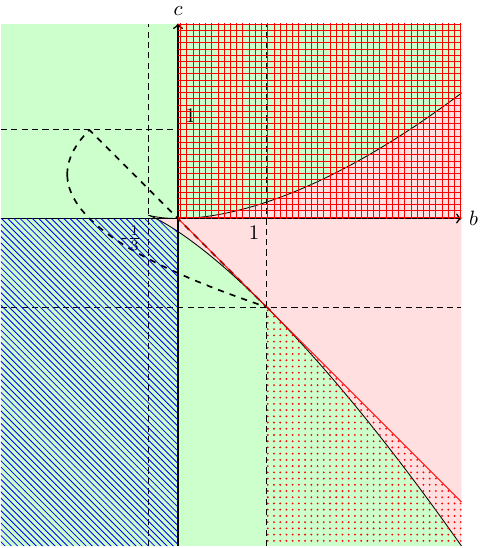}\\
        $a=1$
    \end{minipage}\\
    \begin{minipage}[c]{0.46\linewidth}
        \centering
        \includegraphics[width=0.8\linewidth]{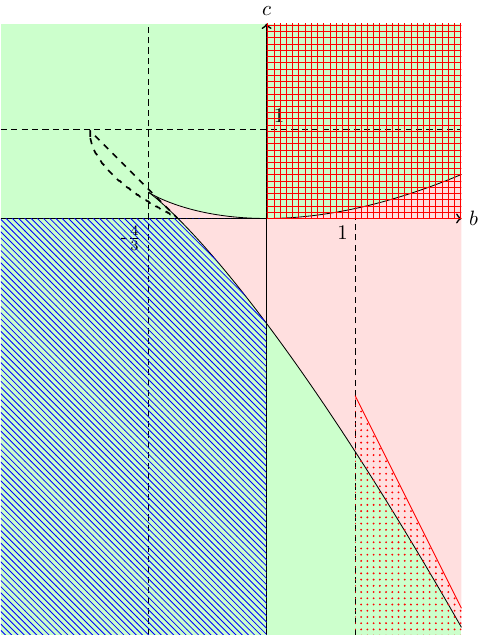}\\
        $a=2$
    \end{minipage}
    \hfill%
    \begin{minipage}[c]{0.46\linewidth}
        \centering
        \includegraphics[width=0.8\linewidth]{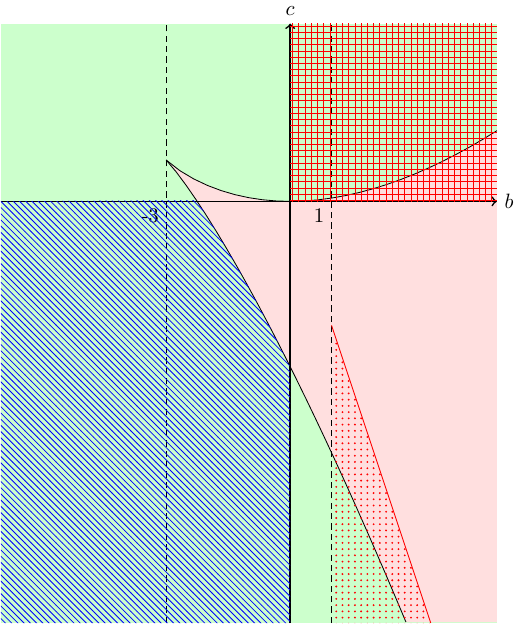}\\
        $a=3$
    \end{minipage}
    \caption{ Two dimensional representation of the areas in $\{b,c\}$ covered by transience and recurrence results for fixed values of $a\in\{0.5,1,2,3\}$. In each figure, the green (resp.red) region
corresponds to $Disc(P ) < 0$ (resp. $Disc(P ) > 0$). The black dashed curve circles the area where the linear recurrence \eqref{linarRec} is stable (note that this
never happens when $a = 3$). The blue lined region corresponds to the area covered
by Theorem \ref{thm:RecurrenceDisc}, and by Theorem \ref{thm:recurrencePHDp} for the case $a = 0.5$. The red dotted region corresponds to Proposition \ref{thm:transience}. The red grid region corresponds to the known result
about transience, discussed in Section \ref{subsec:setting} and proved in Appendix \ref{app:Trans_p}.    }
    \label{fig:disc_2d}
\end{figure}

Returning to our comparison involving deterministic sequences $(\Tilde u_n)_{n \geq 0}$ satisfying \eqref{linarRec}, when $\Disc(P)<0$, since the polynomial $P$ has 2 complex roots, we know that the solutions $(\Tilde    u_n)_{n \geq 0}$ will exhibit oscillatory behaviour, potentially diverging under stronger conditions. Consequently, $(\Tilde    u_n)_{n \geq 0}$ will eventually become negative. This is where the positive part in the definition of $(\Tilde X_n)_{n \geq 0}$ comes into play : by truncating the parameter value of the Poisson random variable to 0, we can expect a stabilizing effect on the asymptotic behaviour of this process.

\subsection{Recurrence}
The following theorem states that it is possible to choose $b$ and $c$ sufficiently negative for the process $X_n$ to be ergodic, for every value of $a$.
\begin{thm}
\label{thm:RecurrenceDisc}
Suppose $b<0$, $c<0$ and $\Disc(P) < 0$. Then $(X_n)_{n \geq 0}$ is a geometrically ergodic Markov chain.
\end{thm}
The proof of Theorem \ref{thm:RecurrenceDisc} is deferred to section \ref{SectionRecurrenceDisc}. 
As a remark, note that in Theorem \ref{thm:RecurrenceDisc}, when $c=0$ we have $\Disc (P) = (a^2+4b)b^2$, making the condition $\Disc(P)<0$ equivalent to the condition $a^2+4b<0$ mentioned in section \ref{theoRecall}.

The following lemma rephrases the condition $\Disc(P)<0$ and will be used several times in the article.
\begin{lemma}
\label{lem_discP}
For $a,b,c\in\R$, the following holds:
\begin{itemize}
    \item if $a^2+3b<0$, then $\text{Disc}(P)<0$ for all values of $c\in\R$.
    \item if $a^2+3b\ge 0$ then  $\text{Disc}(P)<0 \Leftrightarrow c\in ]-\infty, c_-[ \cup ]c_+,\infty[$ where
    $$c_-= \frac{-1}{27}(2a^3+9ab+2 (a^2+3b)^{3/2})\qquad c_+= \frac{-1}{27}(2a^3+9ab-2 (a^2+3b)^{3/2})$$
    Furthermore if $a^2+4b>0$, then $c_-<0<c_+$.
\end{itemize}
\end{lemma}
The proof of this Lemma is postponed in Appendix \ref{app:technique}.

Theorem \ref{thm:RecurrenceDisc} can be interpreted as saying that for any given excitation $a$, possibly much greater than $1$, there exist sufficiently strong inhibitions $b,c < 0$ such that the parameters of the process $(\Tilde X_n)_{n \geq 0}$ satisfy the assumptions of Theorem \ref{thm:RecurrenceDisc}. 


\subsection{Transience}
\label{Transience results}
The following theorem provides some simple sufficient conditions for transience. This partially completements the results of Theorem \ref{thm:RecurrenceDisc}. It shows in particular that the order of parameters is important, introducing an asymmetry in their roles, which is a non trivial property of the process $(\Tilde X_n)_{n \geq 0}$ in view of its initial definition. 
\begin{prop}
\label{thm:transience}
If $a,b,c \in \R$ satisfy one of the following conditions,
\begin{enumerate}
\item $a,b<0$ and $c>1$, or
\item $b>1$ and $ab+c<0$
\end{enumerate}
then the Markov chain $(X_n)_{n\geq0}$ is transient. 
\end{prop}
Proposition \ref{thm:transience} states that when $b$ or $c$ is greater than 1, and the other parameters are negative, $(X_n)_{n \geq 0}$ is transient. To rephrase, there is no inhibition, no matter how strong, that yields a stable process as long as $b$ or $c$ is greater than 1.
 

The proof of Proposition \ref{thm:transience} can be found in sections \ref{SectionTransience1} for case $i)$ and \ref{SectionTransienceb} for case $ii)$.
Let us remark that that second case, the condition $ab+c<0$ is actually related to $\Disc(P)$ being negative as it is precised in Lemma \ref{Lemme ab+c}.

\subsection{Conjecture and numerical illustrations}
\label{Conjecture}

The proof of Theorem \ref{thm:RecurrenceDisc} is based on a Lyapunov function argument following the theory developed by Meyn and Tweedie in the 80's.  
The construction of the Lyapunov function requires $c<0$ while $b<0$ is used to construct an appropriate small set (see section \ref{sec:Lyap}). 
Based on the numerical experiments conducted, we believe that it is possible to slightly relax the assumptions of Theorem \ref{thm:RecurrenceDisc} as follows:
\begin{conjecture}
\label{conjecture_b<1}
If $b \leq 1$ and $c<0$ are such that $\Disc(P) < 0$, then $(X_n)$ is an ergodic Markov chain. 
\end{conjecture}
Note that furthermore, if $a>1, b>1$ and $c<0$ such that $\Disc(P) < 0$, then using properties of the discriminant, it is possible to prove that $ab+c<0$ (see Lemma \ref{Lemme ab+c} in Appendix \ref{app:technique}). As a consequence we can apply Theorem \ref{thm:transience} $ii)$ and deduce that the Markov chain $(X_n)$ is transient. 

Let us now present our numerical simulations to handle the case $a>1$, $b\in[0,1)$ and $c<0$ such that $\Disc(P)=0$. In all the following simulations we chose the initial condition $(X_{-2}, X_{-1},X_0)=(0,0,0)$ and $\lambda=1$.
The state (0,0,0) is of notable interest since it is an accessible state: for any values of $a, b, c$, it is always possible to reach this state from any other state in 3 steps.
We illustrate the fact that the following stopping time :
$$\tau_0 := \inf \{ n \geq 1 ~|~ X_n = (0,0,0) \},$$
satisfies 
\[\begin{cases}
    \P_{(0,0,0)}(\tau_0 = +\infty) = 0\quad\text{ if }\quad b \leq 1, \\
    \P_{(0,0,0)}(\tau_0 = +\infty) > 0\quad\text{ if }\quad b>1.\\
\end{cases}\]

\vspace{0.2cm}
To approach as closely as possible the simulation of $\tau_0$, we instead consider the following random variable, where $n$ is chosen to be large (in the following illustration, $n = 10000$) :
$$\tilde \tau_0 := \tau_0 \wedge n.$$

However, due to numerical reasons, we cannot directly simulate this random variable. This is because, since we expect the Markov chain to be transient when $b>1$, such a direct simulation would yield arbitrarily large values when simulating $(X_n)$, potentially leading to computational errors. To avoid this, we fix a very large threshold denoted by $M$ and simulate the following random variable: 
$$\hat \tau_0 := \Tilde \tau_0 \mathds 1_{(X_n) \text{ did not explode before reaching (0,0,0)}}+ (n+1) \mathds{1}_{(X_n) \text{ exploded}} ,$$
where we considered that the Markov chain $(X_n)$ exploded if there exists $m\leq n$ such that $\Tilde X_m > M$. In the case of an explosion, this amounts to considering $\tau_0 = +\infty$. We then set an arbitrary value, here $n+1$, for $\hat \tau_0$.

\vspace{0.2cm}

We consider a setting where $b$ varies in $[0,4]$ while maintaining a negative discriminant, in order to observe the phase transition at $b=1$. We fixed $a=3$ and $c=-15$ : for this choice of parameters, for any $b \in [0,4], ~ \Disc(P)<0$.  For the sake of conciseness, unless stated otherwise, we will present our numerical study for these given fixed values; however, we have observed consistent behaviour for all tested values of $a$ and $c$ such that $\Disc(P) < 0$ for all $b$ considered.
\vspace{0.2cm}

We first estimate the probability that an excursion starting from $(0,0,0)$ explodes. In Figure \ref{ExplodingProportion} we plot the values obtained by a Monte Carlo estimation with a sample size of $N\approx 10^6$. We provide confidence intervals using Clopper-Pearson intervals which are \textit{exact} confidence intervals for binomial proportion $p$, having a coverage at least equal to $1-\alpha$ for all $p \in (0,1)$. If $0\le X\le N$ denotes the number of successes, this confidence interval is given for $X \ne 0$ and $X\ne N$ by (see Section 2.1 of \cite{thulin_cost_2014}):
$$\Big[B(\alpha/2,X,N-X+1), \quad B(1-\alpha/2,X+1,N-X) \Big],$$
where $B(\gamma,x,y)$ is the $\gamma$-quantile of the $Beta(x,y)$ distribution. If $X = 0$, the interval takes on the form
$$\Big[0,1-(\alpha/2)^{1/N} \Big].$$

\begin{figure}[!h]
\centering
\includegraphics[scale=0.2]{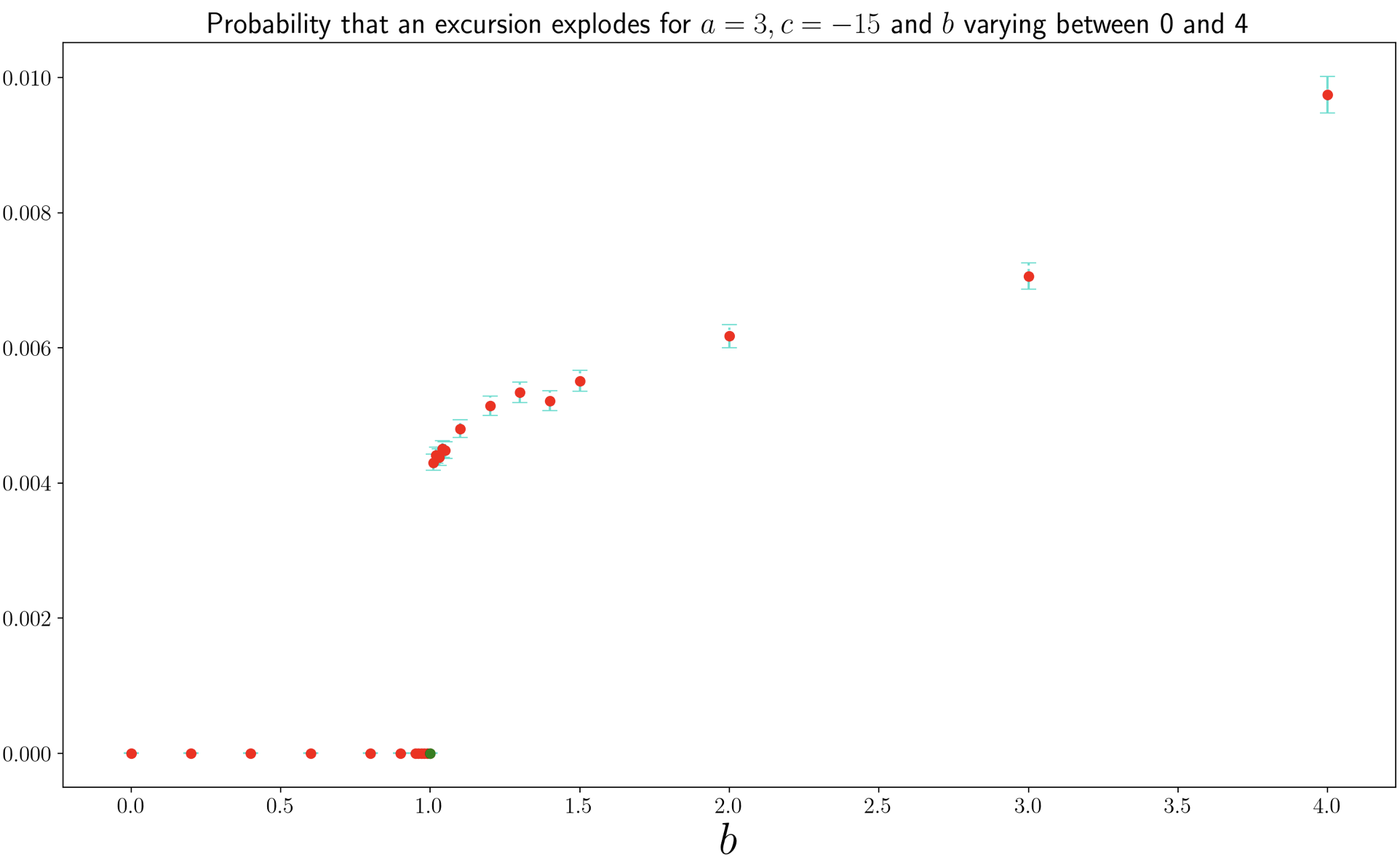}
\caption{Plot (in red) of the proportion of excursions that have exploded relative to the total number of excursions, for some $b \in [0,4]$. For clarity, we put a green dot when $b=1$. Clopper-Pearson type confidence intervals (with $\alpha = 0.01$) around the calculated values have also been plotted in blue. It is noteworthy that the positive proportions exhibit remarkably low values : even within the context where $b>1$, the number of excursions that have exploded remains very small compared to the overall number of excursions.}
\label{ExplodingProportion}
\end{figure}

On Figure \ref{ExplodingProportion}, we observe the transition occurring at $b=1$: for $b\leq 1$, none of the simulated excursions has exploded, whereas there is a strictly positive number of them for $b>1$. Additionally, it is noteworthy that the probability of exploding seems to be increasing in $b$. 

We complete our numerical study by drawing the empirical distribution function of $\hat{\tau}_0$ in Figure \ref{EmpiricalDistribution}. 
\begin{figure}[!h]
\centering
\includegraphics[scale=0.4]{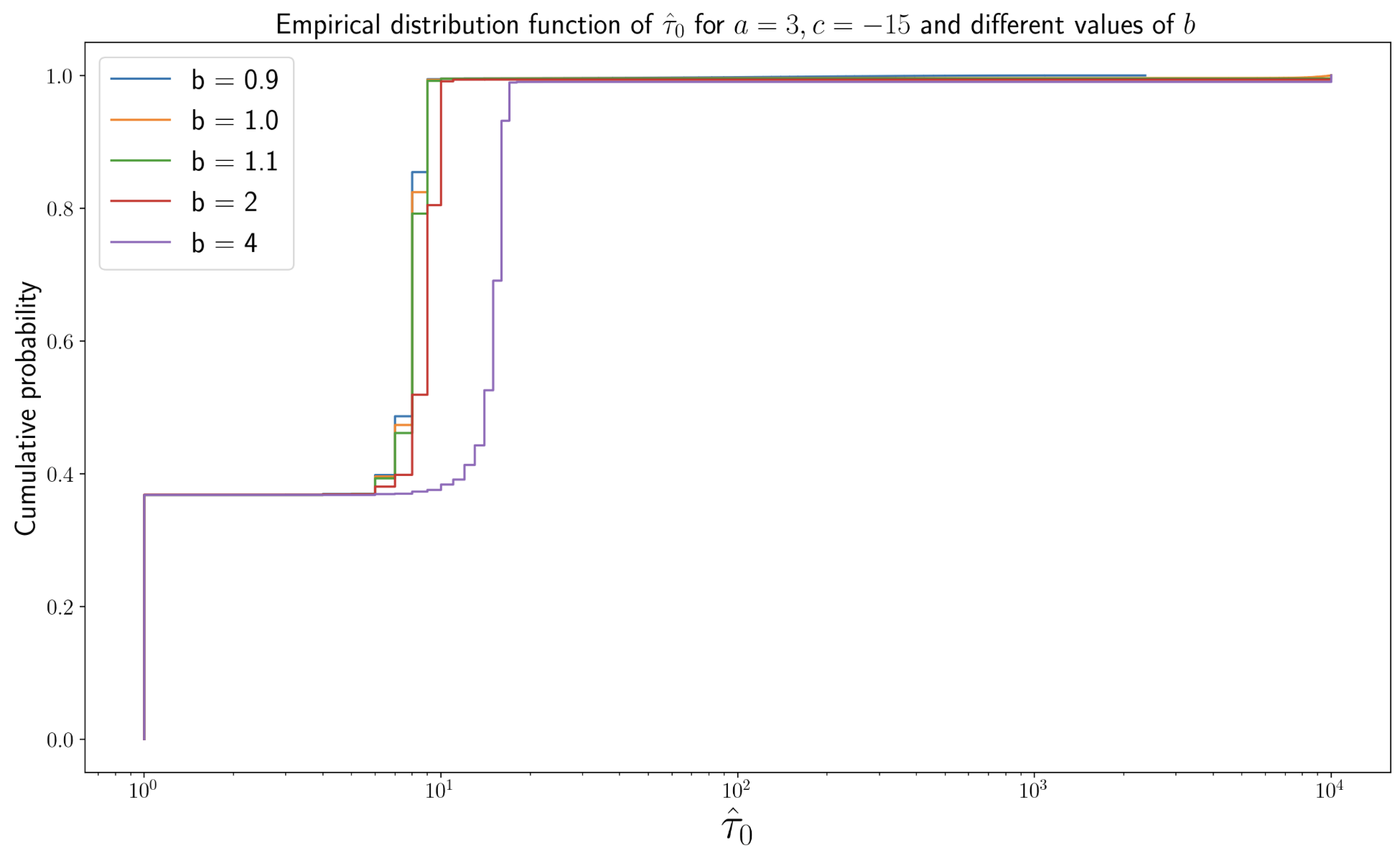}
\caption{Empirical distribution function of the random variable $\hat \tau_0$, for some $b \in [0.9,4]$ (the $\hat \tau_0-$axis is log scaled). For values of $b$ strictly greater than 1, an atom appears at $n+1$: this is a rephrasing of the phenomenon observed in Figure \ref{ExplodingProportion}. In the case of $b=1$, some excursions may attain significant lengths without having exploded.}
\label{EmpiricalDistribution}
\end{figure}
This confirms that as the parameter $b$ increases, the length of excursions $\hat \tau_0$ also increases. For example, for $b=0.9$, the length of excursions never reaches $10^4$. This growth of $\hat \tau_0$ with $b$ appears to become more pronounced as $b$ increases. Nevertheless, we observe that the transition from $b\leq 1$ to $b > 1$ does not drastically change the length of the majority of excursions. However, it enables the observation of excursions of arbitrarily long lengths which happens, as discussed before, in a proportion significantly smaller compared to the total number of excursions considered.

\vspace{0.2cm}

We now put ourselves within a different context in an attempt to explain this very low proportion of excursions that have exploded, even with large values of $b$. We believe that this proportion is strongly dependent on the parameter $a$. Therefore, we conducted a series of simulations where the parameter $a$ varies between 0.1 and 8, while maintaining $\Disc(P)<0$ in order to be consistent with the conjectural framework. To achieve this, we set $b=8$ and $c=-121$. The construction of Figure \ref{ExplodingProportionA} closely follows the procedure outlined for Figure \ref{ExplodingProportion}. 

\begin{figure}[!h]
\centering
\includegraphics[scale=0.4]{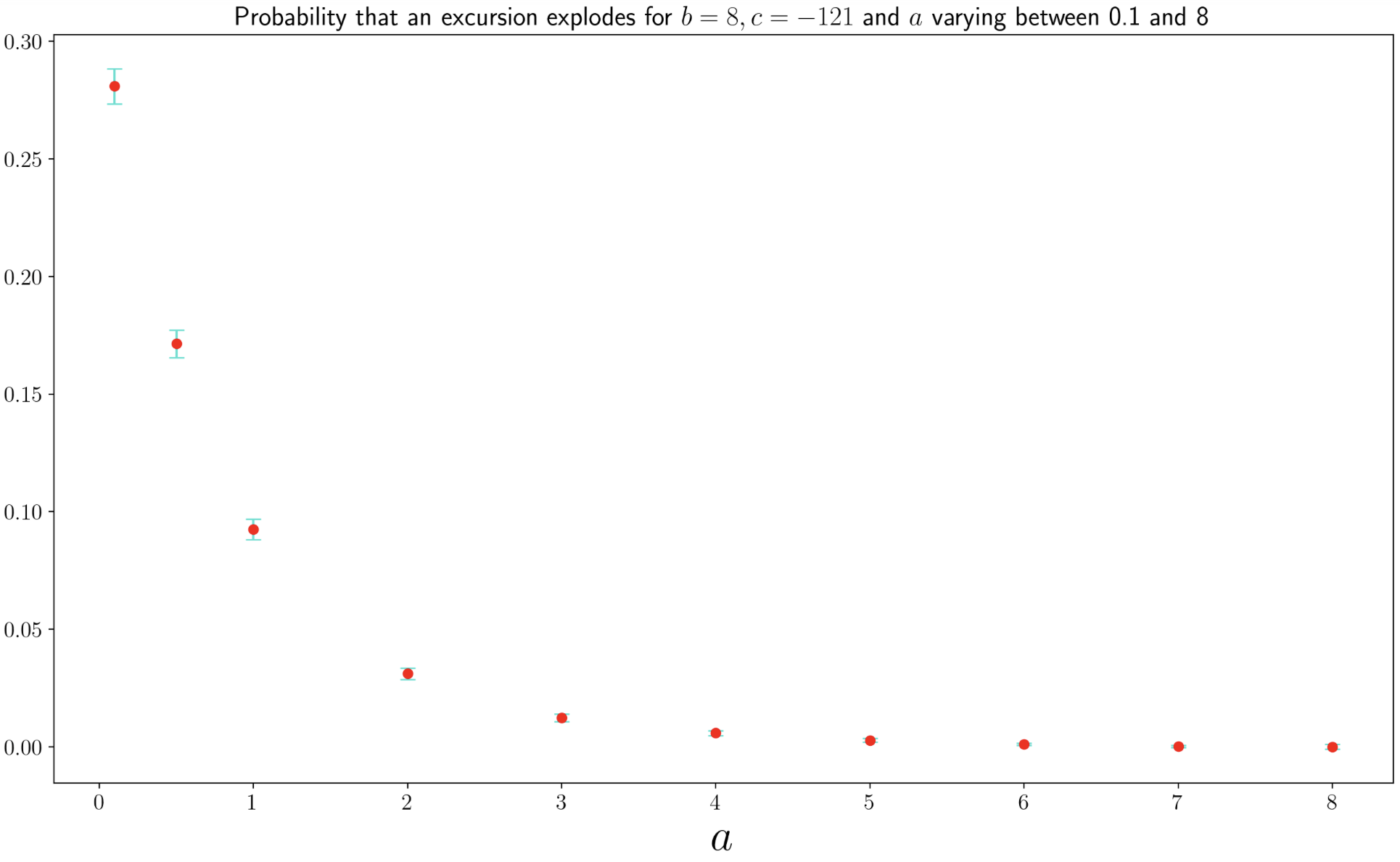}
\caption{Plotting the proportions (and confidence intervals) of excursions that have exploded relative to the total number of excursions. The graph is constructed similarly to Figure \ref{ExplodingProportion}. A strong decrease to 0 is observed, which may be counter-intuitive at first glance.}
\label{ExplodingProportionA}
\end{figure}

\vspace{0.2cm}

To conclude this section of numerical illustrations, we end with a final figure where we use parameters $a=3, b=1.1, c=-15$. We went a step further in analysing exploding excursions for this setting. It seemed that they all followed a similar pattern, alternating between 0 and positive numbers after a certain time. Since $b>1$, the positive numbers tended to increase on average. In the next graph, we randomly picked some excursions of $(\Tilde X_n)$  that exploded with $b=1.1$ as an illustration of our statement. Each color in the graph corresponds to one excursion, traced over the initial thirty values.  

\begin{figure}[!h]
\centering
\includegraphics[scale=0.2]{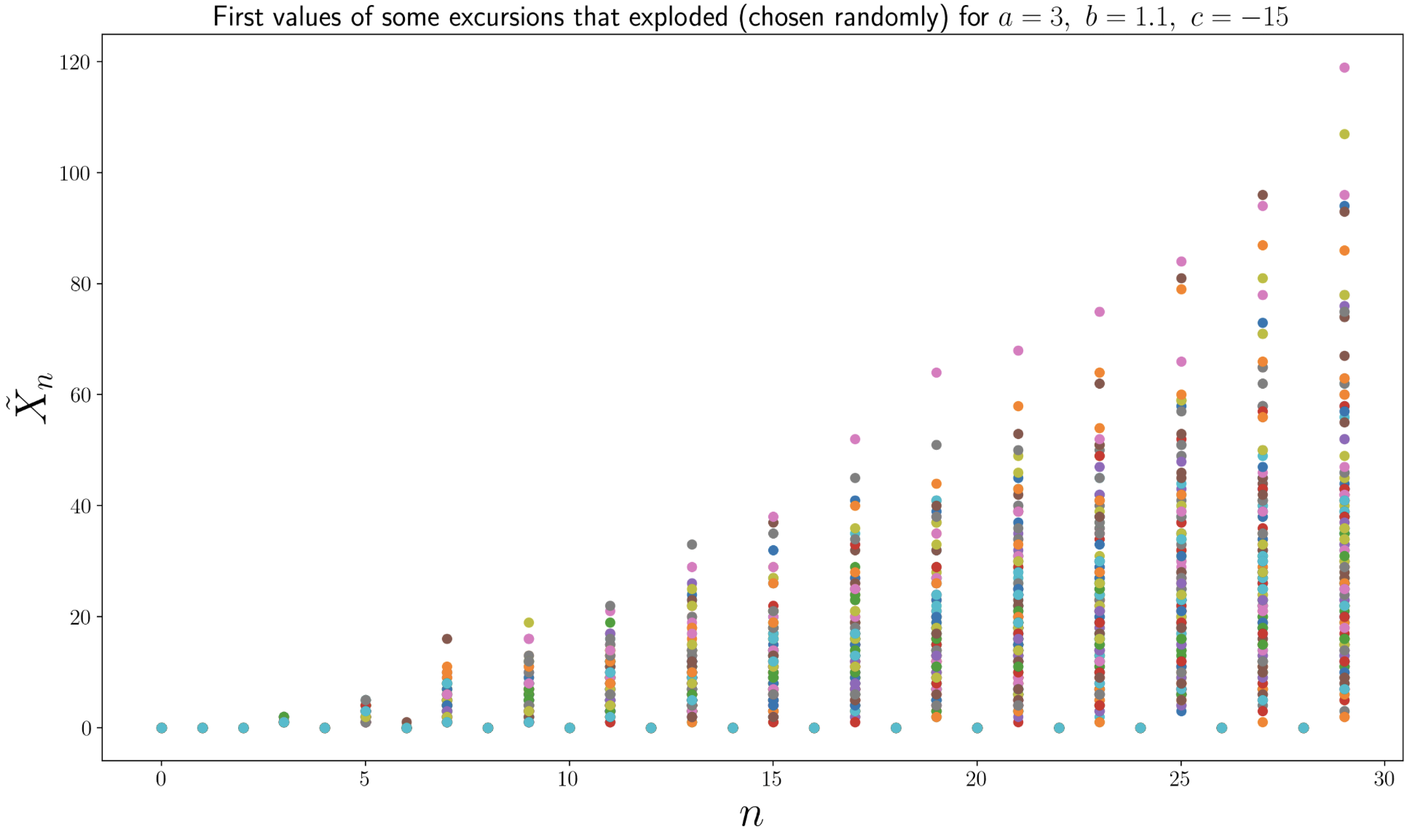}
\caption{Plotting the first 30 values of exploding trajectories, with parameters $a=3, b=1.1, c=-15$. The trajectories all follow a similar pattern, alternating between 0 and positive numbers after a certain time. We remark that this phenomenon is observable for any value of $b>1$.}
\label{ExplodingExcursions}
\end{figure}

On Figure \ref{ExplodingExcursions}, we observe that after $n=6$, each excursion behaves as described: at all even time steps, $\Tilde X_n = 0$ while at all odd time steps, $\Tilde X_n > 0$. We draw the reader's attention to the fact that this type of state $(0,x,0)$ where $x > 0$, precisely forms the sets $S_n$ in the proof of Proposition \ref{thm:transience} for the case $b>1, ab+c < 0$ (see section \ref{SectionTransienceb}). This partially addresses the remark at the end of this section : in practice, we do observe the Markov chain $(X_n)$ reaching these states, although, as seen in Figure \ref{ExplodingProportion}, this rarely occurs.

\section{Recurrence: proofs of Theorems~\ref{thm:recurrencePHDp} and \ref{thm:RecurrenceDisc}}

\subsection{Lyapunov functions and drift criterion}
\label{sec:Lyap}
Drift criteria, initially introduced by Foster \cite{foster_stochastic_1953}, and extensively studied and popularized by Meyn and Tweedie \cite{meyn_markov_2009}, and others, constitute powerful analytical tools. These criteria play a pivotal role in establishing convergence to the invariant measure of Markov chains and determining explicit rates of convergence. We follow the book by Douc, Moulines, Priouret, and Soulier \cite{dmps_markov_chains}.
We begin with the following definitions.

\begin{definition}
\label{DefinitionPetiteSet}
A subset $C\subset \N^p$ is called \emph{petite} \cite[Definition~9.4.1]{dmps_markov_chains}, if there exists a state $x_0\in \N^3$ and a probability distribution $(p_n)_{n\in\N}$ on $\N$ such that
 \[
 \inf_{x\in C}\sum_{n\in\N} p_n P^n(x,x_0) > 0,
 \]
 where $P^n(x,x_0)$ represents the $n$-step transition probability from state $x$ to $x_0$. 
\end{definition}

Recall that, due to the irreducibility of the Markov chain $(X_n)_{n\ge0}$, any finite set is petite (with $x_0$ being the accessible state) and a finite union of petite sets remains petite \cite[Proposition~9.4.5]{dmps_markov_chains}.

\begin{definition}
\label{def:drift_condition}
Let $V : \N^p \to [1,\infty)$ be a function, $\varepsilon\in(0,1]$, $K<\infty$ and $C\subset \N^3$. We say that the \emph{drift condition} $D(V,\varepsilon,K,C)$ is satisfied if
\begin{equation}
\label{Drift Condition}
\forall x \in \N^p, \quad \Delta V(x) := \mathbb{E}_x [V(X_1)-V(X_0)] \leq -\varepsilon V(x)+K\boldsymbol 1_C(x).
\end{equation}
\end{definition}
It is easy to see that \eqref{Drift Condition} implies condition $D_g(V,\lambda,b,C)$ from \cite[Definition~14.1.5]{dmps_markov_chains}, with $\lambda = 1-\varepsilon$ and $b = K$. A function $V$ satisfying a drift condition $D(V,\varepsilon,K,C)$ for some $\varepsilon\in (0,1]$, $K<\infty$ and a petite set $C$ is also called a \emph{Lyapunov function}.

Note that we are in a setting where the Markov chain $(X_n)$ is aperiodic and irreducible. Indeed, since the probability that a Poisson random variable is zero is strictly positive, it is possible to reach the state $\boldsymbol 0 = (0,\ldots,0)$ with positive probability from any state in $p$ steps. In particular, the state $\boldsymbol 0$ is accessible and the Markov chain is irreducible. Furthermore, the Markov chain is aperiodic \cite[Section~7.4]{dmps_markov_chains}, since $P(\boldsymbol 0,\boldsymbol 0) = e^{-\lambda} >  0$. For this setting, in \cite[Theorem~9.4.10]{dmps_markov_chains} it is proven that the notions of \textit{petite} and \textit{small} sets are equivalent. We will thus use the simpler definition of \textit{small} sets, that is, a set $C \subset \N^p$ such that there exists $x_0 \in \N^p$ and $n \in \N$ satisfying :
 \[
 \inf_{x\in C} P^n(x,x_0) > 0.
 \]
 
The following proposition is an easy consequence of results from \cite{dmps_markov_chains}, see \cite[Proposition 3.1]{Costa_Maillard_Muraro_2024}.

\begin{prop}
\label{prop:Lyapounov}
Assume that the drift condition $D(V,\varepsilon,K,C)$ is verified for some $V$, $\varepsilon$, $K$ and $C$ as above and assume that $C$ is petite. Then the Markov chain $(X_n)$ is geometrically ergodic.
\end{prop}

\subsection{Proof of Theorem \ref{thm:recurrencePHDp}}
\label{SectionRecurrencep}



Recall that $p \in \N$ and that the Markov chain $(X_n)$ is defined by 
$$X_n=(\Tilde X_n, \Tilde{X}_{n-1}, \dots, \Tilde X_{n-p+1}), $$
where, knowing $\Big\{\Tilde X_n = x_1, \dots, \Tilde X_{n-p+1} = x_p \Big\}$, $~\Tilde X_{n+1}$ is distributed as a Poisson random variable with parameter 
\begin{equation}
\label{DefinitionParameterS}
\mathbf{s} := (a_1x_1 + \dots + a_px_p  + \lambda)_+.
\end{equation}

Theorem~\ref{thm:recurrencePHDp} is a direct consequence of Proposition \ref{prop:Lyapounov}, together with the following lemma:
\begin{lemma}
If $(a_1)_+ + \dots + (a_p)_+ < 1$, then there exist $\alpha_1, \dots, \alpha_p > 0$ such that the function $V : \N^p \to [1;+\infty)$ defined by
$$V(x_1, \dots, x_p) := \alpha_1 x_1 + \dots + \alpha_p x_p + 1,$$ 
satisfies the drift condition $D(V,\varepsilon,K,C)$ with some $\varepsilon\in (0,1]$, $K<\infty$ and a finite set $C\subset \N^3$.
\end{lemma}
\begin{proof}
Set $\eta \coloneqq 1 - (a_1)_+ - \dots - (a_p)_+> 0$. Define for $1 \le i\le p$
\[
\alpha_i \coloneqq \eta\,\frac{p-i+1}{p} +\sum_{j=i}^p (a_j)_+  >0,
\]
 and \[\alpha_{p+1}:=0,\]
and note that $\alpha_1 = 1$ by definition. Let $V$ be defined as in the statement of the lemma---we claim that it satisfies its conclusion.

Let $\varepsilon > 0$ to be properly chosen later. Then, for all $\vec{x} = (x_1, \dots, x_p) \in \N^p$, we have, from the linearity of $V$ and $\alpha_1=1$
\begin{align*}
\Delta V(\vec{x}) + \varepsilon V(\vec{x}) &= \mathbb{E}_{\vec{x}} [V(X_1)-V(X_0)] + \varepsilon V(\vec{x}) \\ 
&= \alpha_1 \E(\tilde{X_1}\lvert X_0=\vec{x}) + \sum_{i=1}^{p}(\alpha_{i+1} + \alpha_i (\varepsilon - 1)) x_i + \varepsilon.\\
&= \mathbf{s} + \sum_{i=1}^{p}(\alpha_{i+1} + \alpha_i (\varepsilon - 1)) x_i + \varepsilon.
\end{align*}
Using that $\mathbf{s} = (a_1x_1+\cdots+a_px_p+\lambda)_+ \le (a_1)_+x_1 + \cdots + (a_p)_+x_p + \lambda$, we get
\begin{align}
\label{CalculDeltaVp}
\Delta V(\vec{x}) + \varepsilon V(\vec{x}) \le \sum_{i=1}^{p}((a_i)_+ + \alpha_{i+1} + \alpha_i (\varepsilon - 1)) x_i + \varepsilon + \lambda.
\end{align}
The right-hand side of \eqref{CalculDeltaVp} is an affine function of $x_1,\ldots,x_p$. Furthermore, by definition of $\alpha_i$, we have $(a_i)_+ + \alpha_{i+1} - \alpha_i = -\eta/p$, furthermore, $\alpha_i \le \alpha_1 = 1$ for every $i$. Hence, choosing $\varepsilon < \eta/p$, we have that every coefficient in front of the $x_i$'s on the right-hand side of \eqref{CalculDeltaVp} is (strictly) negative. It follows that the set $C \coloneqq \{\vec{x}\in \N^p: \Delta V(\vec{x}) + \varepsilon V(\vec{x}) > 0\}$ is finite. This shows that the drift condition $D(V,\varepsilon, K, C)$ is satisfied, with $K \coloneqq \max_{\vec{x}\in C} \E_{\vec{x}}[V(X_1)-V(X_0)]$, which is finite since $C$ is finite. This proves the lemma.
\end{proof}

\subsection{Proof of Theorem \ref{thm:RecurrenceDisc}}
\label{SectionRecurrenceDisc}

In this section, we are in the setting of Theorem~\ref{thm:RecurrenceDisc}, in particular, we have $p=3$ and we write $a = a_1$, $b=a_2$ and $c=a_3$ the parameters of the process.


\begin{figure}[!h]
    \begin{minipage}[c]{0.46\linewidth}
        \centering
        \includegraphics[scale=0.25]{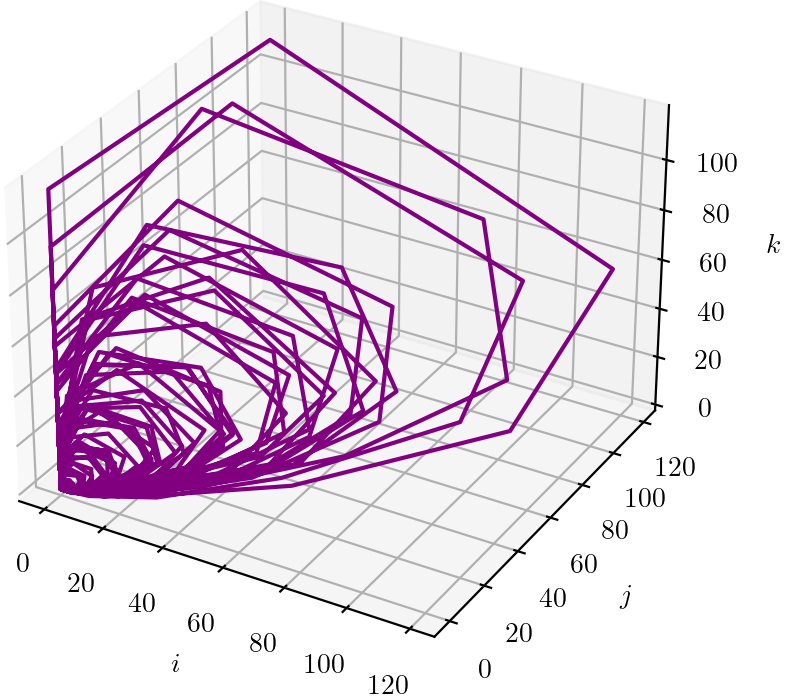}
    \end{minipage}
    \hfill%
    \begin{minipage}[c]{0.46\linewidth}
        \centering
        \includegraphics[scale=0.25]{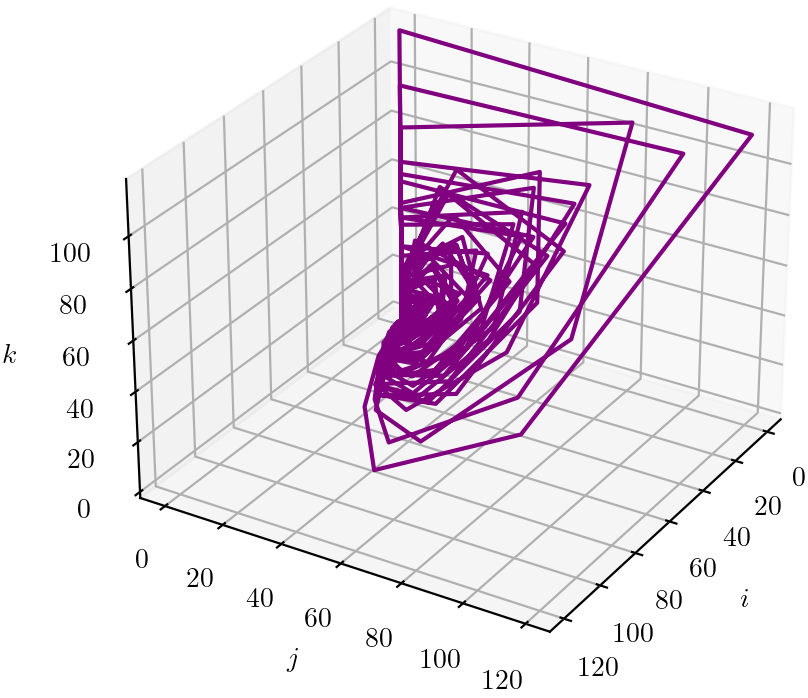}
    \end{minipage}
    \caption{Typical trajectory of $(X_n)$, for parameters satisfying assumptions of Theorem \ref{thm:RecurrenceDisc}. Here, the parameters are $a = 2.5, b = -1, c = -3$. For these parameters, we have $\Disc(P) = -188.25$.}
\end{figure}

\subsubsection*{Construction of a small set.}
In order to use Proposition \ref{prop:Lyapounov} we have to identify a small set, and prove we can build an appropriate Lyapunov function. We have the following lemma.
\begin{lemma}
Let us define the following set of states: 
$$A := \{(i,j,k) \in \N^3, ~ ai + bj + ck + \lambda \leq 0\}.$$
Assuming that $b,c\le 0$, then $A$ is a \textit{small} set.
\label{AisSmall}
\end{lemma}

\begin{proof}
By definition of $A$, we have $s_{ijk} = 0$ for all $(i,j,k)\in A$, hence $P((i,j,k),(0,i,j)) = 1$. Furthermore, for every $i\in \N$, since $b,c\le 0$, we have
\[
P((0,i,j),(0,0,i)) = e^{-s_{0ij}} = e^{-(\lambda+bi+cj)_+} \ge e^{-\lambda}.
\]
And since $c\le 0$,
\[
P((0,0,i),(0,0,0)) = e^{-s_{00i}} = e^{-(\lambda+ci)_+} \ge e^{-\lambda}.
\]
It follows that 
$$
\inf_{(i,j,k)\in A} P^3((i,j,k),(0,0,0)) \ge e^{-2\lambda} > 0,
$$
which shows that $A$ is \textit{small}.
\end{proof}

\subsubsection*{Construction of a Lyapunov function.}

For $\alpha \ge 0$, define $V_\alpha : \N^3 \to [1;+\infty)$ as follows:
\begin{equation}
V_\alpha(i,j,k) := \dfrac{i + \alpha j}{j + \alpha k + 1} + 1.
\label{functionV}
\end{equation}

We will prove the following:
\begin{lemma}
Under the hypotheses of Theorem~\ref{thm:RecurrenceDisc}, there exists $\alpha > 0$, $\varepsilon>0$, $K<\infty$ and a finite set $B\subset \N^3$, such that the function $V_\alpha$ defined by \eqref{functionV} satisfies the drift condition $D(V_\alpha,\varepsilon,K,C),$ with $C = A \cup B$.
\label{VisLyapunov}
\end{lemma}
\begin{proof}

Let $(i,j,k) \not \in A$, so that $s_{ijk} = ai+bj+ck+\lambda > 0$. We compute : 
\begin{align*}
\Delta V_\alpha(i,j,k) &= \sum_{\ell=0}^{+\infty} \dfrac{s_{ijk}^\ell e^{-s_{ijk}}}{\ell !} V_\alpha(\ell,i,j) - V_\alpha(i,j,k) \\
&= \sum_{\ell=0}^{+\infty} \dfrac{s_{ijk}^\ell e^{-s_{ijk}}}{\ell !}\dfrac{\ell + \alpha i}{i + \alpha j + 1} - \dfrac{i + \alpha j}{j + \alpha k + 1} \\
&= \dfrac{(s_{ijk} + \alpha i)(j+\alpha k + 1) - (i+\alpha j)(i+\alpha j + 1)}{(i+\alpha j + 1)(j+\alpha k + 1)} \\
&= \dfrac{q(i,j,k) + L(i,j,k)}{(i+\alpha j + 1)(j+\alpha k + 1)},
\end{align*}
where $q$ is a quadratic form defined by
$$q(i,j,k) = - i^2 + (b-\alpha^2)j^2 + c\alpha k^2 + (a- \alpha)ij + \alpha(a+\alpha)ik + (c+b\alpha)jk,$$
and $L$ is a linear form.

Note that if we consider $(i,j,k)$ to be elements of $\R^3$, then $q$ is a real quadratic form in the sense of the definitions given in the first section of this paper. We will thus use the theoretical framework as if we worked on $\R^3$, but note that our space of interest, which is $\N^3$, is far smaller. 

\underline{\textbf{Step 1:} Choice of an appropriate $\alpha$}\\
Note that $q$ is a quadratic form of degree 3, having the following representation matrix in the canonical base : 
$$M_\alpha := \begin{pmatrix}
-1 & \dfrac{a-\alpha}{2} & \dfrac{\alpha(a+\alpha)}{2} \\
\dfrac{a- \alpha}{2} & b - \alpha^2 & \dfrac{c+b\alpha}{2} \\
\dfrac{\alpha(a+\alpha)}{2} & \dfrac{c+b\alpha}{2} & c\alpha \\
\end{pmatrix}.$$
Denote by $R(\alpha)$ the first minor of order 2 of $M_\alpha$, that is, 
\begin{equation}
    \label{def:R}
R(\alpha) = \begin{vmatrix}
-1 & \dfrac{a-\alpha}{2} \\
\dfrac{a- \alpha}{2} & b - \alpha^2
\end{vmatrix} =\frac{1}{4}( -a^2-4b+2a\alpha+3\alpha^2).
\end{equation}
The Gauss-Lagrange reduction algorithm for $q$ yields, assuming that $R(\alpha) \ne 0$:
\begin{equation}
q(i,j,k) = -\left( i+\dfrac{\alpha-a}{2}j-\dfrac{\alpha(a+\alpha)}{2}k\right)^2 - R(\alpha) \left( j - \dfrac{K(\alpha)}{2 R(\alpha)} k \right)^2 + \dfrac{\det M_\alpha}{R(\alpha)} k^2,
\label{GaussReduction}
\end{equation}
where 
\begin{equation}
    \label{def:K}
    K(\alpha) := c + \alpha b - \dfrac{\alpha(\alpha+a)(\alpha-a)}{2}.
\end{equation}
A straightforward but tedious calculation gives : 
$$\det M_\alpha =  \dfrac{1}{4} (\alpha^3+a\alpha^2-b\alpha+c)^2 = \dfrac{1}{4} Q(\alpha)^2 \ge 0,$$
where the polynomial $Q(X)$ is defined by 
$$Q(X) := X^3+aX^2-bX+c = -P(-X) .$$
Recall that we assumed that $\Disc(P) < 0$, so $P$ has a unique real root, hence $Q$ as well. Denote by $\alpha_Q$ this root. In particular, this gives that $\det M_{\alpha_Q} = 0$. Since the dominant coefficient of $Q$ is positive, and $Q(0) = c < 0$, we have that $\alpha_Q > 0.$\medskip

Let us now prove that with the choice of $\alpha=\alpha_Q$ the quadratic form $q$ is negative, which is equivalent to $R(\alpha_Q)>0$ thanks to the decomposition \eqref{GaussReduction}. 
\begin{lemma}
\label{lem:R(alpha_Q)}
Under the assumptions of Theorem \ref{thm:RecurrenceDisc}, if $\alpha_Q$ is the unique real root of $Q$, then $R(\alpha_Q)>0$. 
\end{lemma} 
The very computational proof of this lemma is postponed to Appendix \ref{app:technique}.\medskip

\underline{\textbf{Step 2:} The isotropic line of the quadratic form.} \\ 
From now on, we set $\alpha = \alpha_Q$, so that $\det M_{\alpha_Q} = 0$ and $R(\alpha_Q) > 0$ by Lemma~\ref{lem:R(alpha_Q)}. Equation \eqref{GaussReduction} shows that $q(\vec x) \le 0$ for all $\vec x\in \R^3$. We seek to determine the \emph{isotropic vectors} of $q$, i.e. the vectors $\vec x\in\R^3$ such that $q(\vec x) = 0$. By \eqref{GaussReduction}, we have 
\[
q(\vec x) = 0 \Leftrightarrow x\in \text{Vect}(\vec x^*),\text{ where }\vec x^* \coloneqq \left( \dfrac{\alpha_Q(a+\alpha_Q)}{2} + \dfrac{K(\alpha_Q)(a-\alpha_Q)}{4R(\alpha_Q)}, \dfrac{K(\alpha_Q)}{2R(\alpha_Q)}, 1 \right)^T.
\]
We claim that $\mathrm{Vect}(\vec x^*)\cap (\R_+)^3 = \{\boldsymbol 0\} = \{(0,0,0)\}$. For this, it is sufficient to prove that two of the coordinates of $\vec x^*$ have opposite signs. We will thus show that $\dfrac{K(\alpha_Q)}{2R(\alpha_Q)} < 0$.
Since $\alpha_Q$ satisfies $R(\alpha_Q)>0$, we have 
$$3 \alpha_Q^2 + 2a\alpha_Q -a^2-4b > 0,$$ 
so that
$$a^2+2b < 3\alpha_Q^2+2a\alpha_Q-2b,$$
and therefore,
\begin{align*}
2K(\alpha_Q)= -\alpha_Q^3 + (a^2+2b)\alpha_Q + 2c &< -\alpha_Q^3+(3\alpha_Q^2+2a\alpha_Q-2b)\alpha_Q+2c \\
&= 2(\alpha_Q^3+a\alpha_Q^2-b\alpha_Q+c) \\
&= 2 Q(\alpha_Q) = 0.
\end{align*}
This yields $\dfrac{K(\alpha_Q)}{2R(\alpha_Q)} < 0$ and thus $\mathrm{Vect}(\vec x^*)\cap (\R_+)^3 = \{\boldsymbol 0\}$.

\underline{\textbf{Step 3: a continuity argument}} 
Let us now prove that there is only a finite number of $(i,j,k) \in \N^3\setminus A$ such that $\Delta V_{\alpha_Q}(i,j,k) + \varepsilon V_{\alpha_Q}(i,j,k) > 0$, for $\varepsilon>0$ small enough.

We have
$$\Delta V_{\alpha_Q}(x,y,z) + \varepsilon V_{\alpha_Q}(x,y,z) = \dfrac{q(x,y,z) + \varepsilon \Tilde q(x,y,z) + L(x,y,z) + \varepsilon \Tilde L(x,y,z)}{(x+\alpha_Q y + 1)(y+\alpha_Q z + 1)},$$
where we recall that $\alpha_Q>0$ and where
$$\Tilde q(x,y,z) = x^2+\alpha_Q(\alpha_Q+1)y^2+ (2\alpha_Q + 1)xy + \alpha_Q xz + \alpha_Q^2yz$$
and $\tilde{L}$ is a linear functional in $(x,y,z)$.

Define 
$
S\coloneqq \R_+^3 \cap \{\vec x\in \R^3: \|\vec x\| = 1\}.
$
From Step 2, we deduce that $q(\vec x) < 0$ for all $\vec x\in S$. Since $S$ is compact and $q$ is continuous, it follows that there exists a positive constant $C$ such 
\[ \forall \vec s\in S: q(\vec x)\le -2C\]
We deduce that for a second constant $\tilde{C}$, for $\vec x\in S$,
\[q^\varepsilon(\vec x)=  q(\vec x) + \varepsilon \Tilde q(\vec x) \le -2C +\varepsilon \tilde{C}.\]
In what follows, we choose $\varepsilon$ small enough such that $q^\varepsilon(\vec x) \le -C$ for all $\vec x\in S$.
Note that for any $(x,y,z)\in(\R_+)^3$, by homogeneity of $q^\varepsilon$,
 \begin{align*}
     \Delta V_{\alpha_Q}(x,y,z) + \varepsilon V_{\alpha_Q}(x,y,z) 
     \le \frac{-C \|(x,y,z)\|_2^2 + L(x,y,z) + \varepsilon \Tilde L(x,y,z)}{(x+\alpha_Q y + 1)(y+\alpha_Q z + 1)},
 \end{align*}
Since $L$ and $\tilde L$ are linear functionals, there are only a finite number of states $(i,j,k)$ satisfying $(i,j,k) \in \N^3\setminus A$ such that $\Delta V_{\alpha_Q}(i,j,k) + \varepsilon V_{\alpha_Q}(i,j,k) > 0$. Let us denote by $B$ this set.

Now for every $(i,j,k) \in A$, we have
\begin{align*}
\Delta V_{\alpha_Q}(i,j,k) + \varepsilon V_{\alpha_Q}(i,j,k) &\leq \E_{(0,i,j)}[V_{\alpha_Q}(X_1)] \\
&= V_{\alpha_Q}(0,i,j) \\
&= \dfrac{\alpha i}{i + \alpha j + 1} + 1 \\
&\leq 2 \vee (\alpha + 1) =: \Tilde K < +\infty
\end{align*}

We can now conclude that the drift condition $D(V_{\alpha_Q},\varepsilon, K, A \cup B)$ is satisfied. This finishes the proof of Lemma~\ref{VisLyapunov}.
\end{proof}

By Lemma \ref{AisSmall}, the set A is a small set. Since $C$ is a finite set, we have that $A \cup C$ is a small set
too. This concludes the proof of Theorem \ref{thm:RecurrenceDisc}, using Lemma~\ref{VisLyapunov} and Proposition~\ref{prop:Lyapounov}.

\subsubsection*{Remark.}
In the case where $b \in (0,1)$ and $\Disc(P)<0$, the function $V$ defined by \eqref{functionV} remains a Lyapunov function. However, in this situations, the set $A$ is no longer small. For a deeper discussion of this case, please refer to Section \ref{Conjecture}.

\section{Proofs of Transience cases}
In this section, we prove Proposition \ref{thm:transience}. To accomplish this, we will systematically employ a methodology analogous to our previous approach (see section 4 in \cite{Costa_Maillard_Muraro_2024}), using the following lemma adapted for the present three-parameter process :

\begin{lemma}
\label{lem:strategy_transience}
Let $S_1, S_2,\ldots$ be a sequence of subsets of $\N^3$ and $0<m_1< m_2<\dots$ an increasing sequence of integers. Suppose that
\begin{enumerate}
    \item On the event $\bigcap_{n\ge 1} \{X_{m_n} \in S_n\}$, we have $X_n \ne (0,0,0)$ for all $n\ge 1$,
    \item $\P_{(0,0,0)}(X_{m_1}\in S_1) > 0$ and for all $n\ge 1$ and every $x\in S_n$, we have $\P_x(X_{m_{n+1}-m_n} \in S_{n+1}) > 0$.
    \item There exist $(p_n)_{n\ge 1}$ taking values in $[0,1]$ and such that $\sum_{n\ge 1} (1-p_n) < \infty$, such that
    \[
    \forall n\ge 1: \forall x\in S_n: \P_x(X_{m_{n+1}-m_n} \in S_{n+1}) \ge p_n.
    \]
\end{enumerate}
Then the Markov chain $(X_n)_{n\ge0}$ is transient.
\end{lemma}
\begin{proof}
Since $(0,0,0)$ is an accessible state, it is enough to show that
\[
\P_{(0,0,0)}(X_n \ne (0,0,0)\,,\forall n\ge 1) > 0.
\]
Using assumption 1, it is sufficient to prove that
\begin{equation}
\label{eq:strategy_transience_1}
\P_{(0,0,0)}(X_{m_n} \in S_n\,,\forall n\ge 1) > 0.
\end{equation}
By assumption 3, there exists $n_0\ge 1$ such that $\prod_{n\ge n_0} p_n > 0$. It follows that for every $x\in S_{n_0}$,
\[
\P_x(X_{m_n-m_{n_0}}\in S_{n}\,,\forall n>n_0) \ge \prod_{n\ge n_0} p_n > 0.
\]
Furthermore, by assumption 2, we have that
\[
\P_{(0,0,0)}(\forall n\le n_0\,, X_{m_n}\in S_n) > 0.
\]
Combining the last two inequalities yields \eqref{eq:strategy_transience_1} and finishes the proof.
\end{proof}

\subsection[\texorpdfstring{Case $a,b<0, c>1$}{Case a,b<0, c>1}]{Case $a,b<0, c>1$}
\label{SectionTransience1}

In this region of parameters, the Markov chain eventually reaches one of the axes (see Figure \ref{fig:transience1}).

\begin{figure}[!h]
    \begin{minipage}[c]{0.46\linewidth}
        \centering
        \includegraphics[scale=0.25]{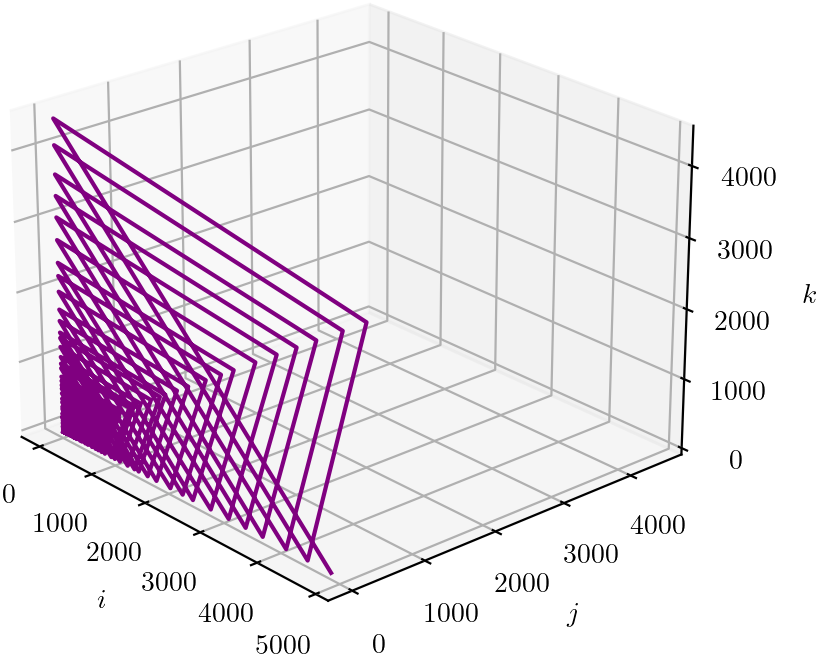}
    \end{minipage}
    \hfill%
    \begin{minipage}[c]{0.46\linewidth}
        \centering
        \includegraphics[scale=0.25]{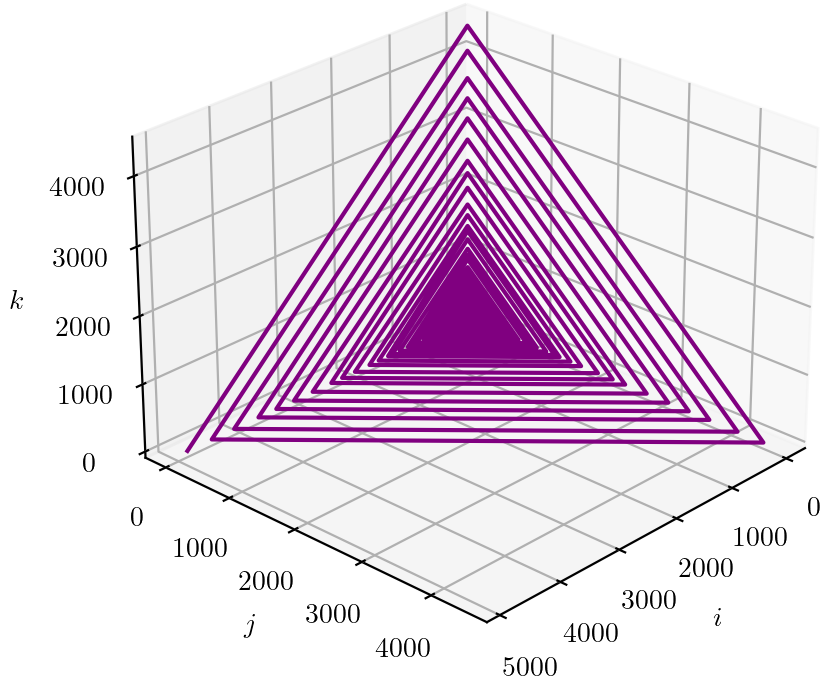}
    \end{minipage}
    \caption{Typical trajectory for the case $a,b<0, c>1$. Here, the parameters are $a = -1, b = -1, c = 1.1$.}
    \label{fig:transience1}
\end{figure}
Indeed, given that both $a$ and $b$ are negative, when the Markov chain $(X_n)$ reaches a state $(i,0,0)$ with $i \geq \max \left\{ -\frac{\lambda}{a}, -\frac{\lambda}{b} \right\}$, the fact that $s_{i00} = (ai+\lambda)_+ = 0$ and $s_{0i0} = (bi+\lambda)_+ = 0$, implies that the subsequent two steps of the Markov chain will lead to states $(0,i,0)$ and $(0,0,i)$. Afterwards, the Markov chain will hit the state $\left( \mathcal{P}(ci+\lambda),0,0 \right)$, where $ci+\lambda > i$. 

Consequently, focusing on the $i$ axis as an illustration, starting from $(k,0,0)$ with a sufficiently large $k$, the Markov chain $(X_n)$ is highly likely to return to a state $(k',0,0)$ where $k'$ satisfies $k' > k$, along the $i$ axis, within three steps.

To formalize these observations, it is natural to consider the Markov chain generated by the third power of the transition matrix, denoted as $P^3$, namely $(X_{3n+1})_{n\geq 0}$. 

For $i \geq \max \left\{ -\frac{\lambda}{a}, -\frac{\lambda}{b} \right\},$ we have $s_{i00}= s_{0i0} = 0$ and thus  : 
\begin{equation}
\label{poisson law}
\begin{aligned}
\mathbb{P}\left( X_{3n+3} = (k,0,0) ~\big| ~ X_{3n} = (i,0,0) , i \geq \max \left\{ -\frac{\lambda}{a}, -\frac{\lambda}{b} \right\} \right) &= \dfrac{e^{-s_{00i}} s_{00i}^k}{k!} \\
&= \dfrac{e^{-(ci+\lambda)} (ci+\lambda)^k}{k!}.
\end{aligned}
\end{equation}

Equation \eqref{poisson law} means that if $\Tilde X_{3n} \geq \max \left\{ -\frac{\lambda}{a}, -\frac{\lambda}{b} \right\}$, and $\Tilde X_{3n-1} = \Tilde X_{3n-2} =0$, then $\Tilde X_{3n+1} = \Tilde X_{3n+2} = 0$, and $\Tilde X_{3n+3}$ is a Poisson random variable with parameter $c\Tilde X_{3n} + \lambda$.

We will now prove our statement.

\begin{proof}[Proof of the transience of $(X_n)$ when $a,b<0$ and $c>1$]

Fix $r \in (1,c)$. We wish to apply lemma \ref{lem:strategy_transience} with 
$$m_n = 3n-1, \quad n \geq 1$$
and 
$$S_n = \{(i,0,0) \in \N^3 : i \geq r^n\}.$$
Assumption (1) holds immediately since if $X_{3n-1} = (i,0,0) \in S_n$, then $X_{3n} = (j,i,0)$ and $X_{3n+1} = (k,j,i)$ for some $j,k \in \N$, hence $(X_{3n-1},X_{3n},X_{3n+1}) \ne (0,0,0)$.

We now verify that the second assumption holds.
For states $x,y\in \N^3$, write $x\to_1 y$ if $\P_x(X_1 = y) > 0$.  Furthermore, for $S\subset \N^3$, write $x\to_1 S$ if $x\to_1 y$ for some $y\in S$. Note that $(0,0,0) \to_1 (i,0,0)$ for every $i\in\N$, so that $(0,0,0)\to_1 S_1$. Now, for every $i\in \N$, we have $(i,0,0)\to_1 (0,i,0)$, then, because $b<0$, $(0,i,0)\to_1 (0,0,i)$ and then since $c>0$, $(0,0,i)\to_1 (j,0,0)$ for every $j\in\N$. In particular, from every $x\in S_n$, we can indeed reach $S_{n+1}$ in three steps. Hence, the second assumption is verified as well.

We now prove the third assumption. We claim that there exists $n_0 \in\N$, such that the following holds:
\begin{equation}
\label{eq:toshow1}
    \forall n\geq n_0, \forall x\in S_n: \P_x(X_3\in S_{n+1}) \ge 1 - \dfrac{c}{(c-r)^2r^{n}}.
\end{equation}
To prove \eqref{eq:toshow1}, first note that according to the earlier remark on \eqref{poisson law}, if $n_0$ is chosen such that $r^{n_0} \geq \max \left\{ -\frac{\lambda}{a}, -\frac{\lambda}{b} \right\}$, then starting from a state $(x_0,0,0)$ with $x_0 \geq r^{n_0}$, we have $\Tilde{X}_1 = \Tilde X_2=0$ almost surely and $\Tilde{X}_{3}\sim \mathcal{P}(c\Tilde X_0 + \lambda)$. Therefore, if $n\ge n_0$ and $x_0 \ge r^n \ge r^{n_0}$,
\begin{align*}
1 - \P_{(x_0,0,0)}(\Tilde X_3 \ge r^{n+1}, \Tilde X_1 = \Tilde X_2 = 0) 
&= \P_{(x_0,0,0)}(\Tilde X_3 < r^{n+1})\\
&\le \proba{\mathcal{P}(cx_0 + \lambda)<r^{n+1}}\\
&\leq \proba{\mathcal{P}(cr^{n}) <r^{n+1}} \\
&= \proba{\mathcal{P}(cr^{n})-cr^{n} <r^{n}(r-c)} \\
&=\proba{|\mathcal{P}(cr^{n})-cr^{n}| > r^{n}(c-r)} \\
&\leq \dfrac{c}{(c-r)^2r^{n}},
\end{align*}
by the Bienaymé-Chebychev inequality. This proves \eqref{eq:toshow1}, which implies,
$$\forall x\in S_n: \P_x(X_3\in S_{n+1}) \ge p_n \coloneqq \left(1 - \dfrac{c}{(c-r)^2r^{n}}\right)_+,$$
and
$$\sum_{n\ge 1} (1-p_n) \le \sum_{n\ge1} \dfrac{c}{(c-r)^2r^{n}} < \infty.$$
This proves that the third assumption of Lemma~\ref{lem:strategy_transience} holds which shows that the Markov chain is transient.
\end{proof}

\subsection[\texorpdfstring{Case $b>1, ab+c < 0$}{Case b>1, ab+c<0}]{Case $b>1, ab+c < 0$}
\label{SectionTransienceb}

\begin{proof}[Proof of the transience of $(X_n)$ when $b>1$ and $ab+c<0$]

We will prove the result using exactly the same method as in the previous case, that is, we will employ Lemma 4, with modifications to the definitions of $m_n$ and $S_n$ to accommodate the current case of interest. Before going into the proof details, let us provide some intuition regarding the behaviour of $(X_n)$ in this case. 
\begin{figure}[!h]
    \begin{minipage}[c]{0.46\linewidth}
        \centering
        \includegraphics[scale=0.25]{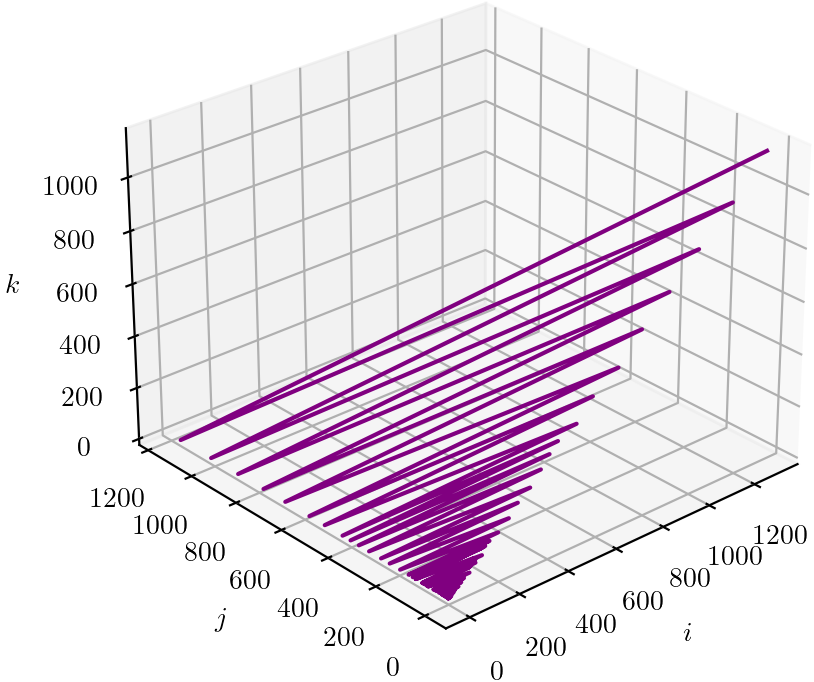}
    \end{minipage}
    \hfill%
    \begin{minipage}[c]{0.46\linewidth}
        \centering
        \includegraphics[scale=0.25]{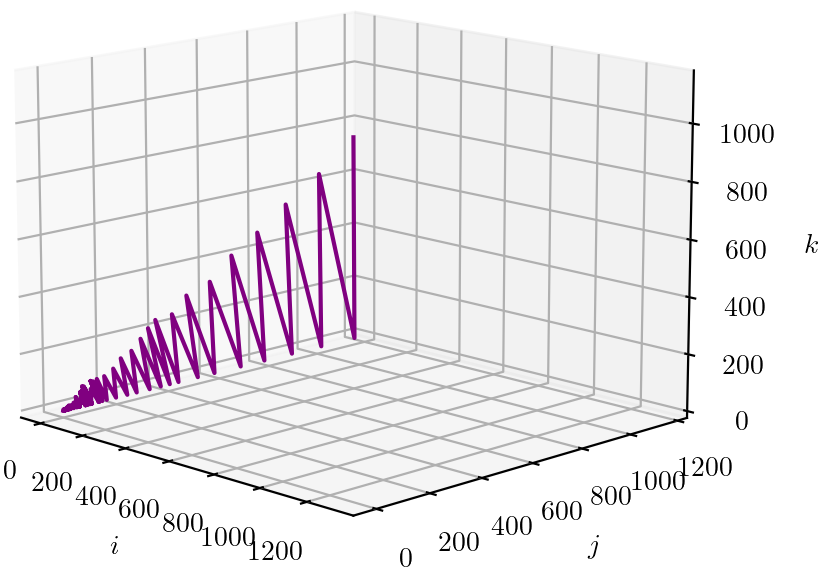}
    \end{minipage}
    \caption{Typical trajectory of $(X_n)$ for the second case of Proposition \ref{thm:transience}. Here, $a=-1, b=1.1, c=0.5$.}
    \label{fig:transience2}
\end{figure}

Let's assume we start from a state $X_0 = (0,x,0)$ where $x>0$. We would then expect, on average, the next step of $(X_n)$ to be $X_1 = (bx,0,x)$. Continuing the reasoning : the next would, on average, be $X_2 = \big((ab+c)x,bx,0\big)$. Now since $ab+c<0$, we end up with $X_2 = (0,bx,0)$ (recall that we take the positive part of the Poisson random variable in the definition of $(X_n)$). To conclude, as $b>1$, on average the Markov chain moves in two steps from a state $(0,x,0)$ to a state $(0,y,0)$ with $y>x$. See Figure \ref{fig:transience2}.

Let us now formalize this proof outline. For this purpose, let $r \in (1,b)$ such that $$ab+c+b-r < 0,$$ and let us denote by $\varepsilon$ the quantity : 
$$\varepsilon := b-r > 0.$$ 
We also define :
$$\beta := \dfrac{-a-a\lambda}{ab+c+\varepsilon}.$$
As announced, we will use Lemma \ref{lem:strategy_transience} with 
$$m_n := 2n-1, \quad n \geq 1$$
and 
$$S_n := \{(0,i,0) \in \N^3 : i \geq r^n\}.$$
For the same reasons as in the previous case, Assumption (1) holds : indeed, if $X_{2n-1} = (0,i,0) \in S_n$, then $X_{2n} = (j,0,i)$ for some $j$, hence $X_{2n-1},X_{2n} \ne (0,0,0)$.

For the second assumption of lemma \ref{lem:strategy_transience}, note that, for every $i,j\in \N$, since $b>0$, we have $(0,i,0)\to_1 (j,0,i)$. Then, because the probability for a Poisson random variable to be zero is positive, we have $(j,0,i)\to_1 (0,j,0)$ for every $j\in\N$. Particularly, for every $x \in S_n$, we can indeed reach $S_{n+1}$ in two steps. Therefore, the second assumption is also satisfied.

We now establish the third assumption, asserting that there exists $n_0 \in \N$ such that the following holds :
\begin{align}
\label{eq:toshow2}
    \forall n\geq n_0, \forall x\in S_n: \P_x(X_2\in S_{n+1}) \ge 1 - \dfrac{b+\lambda/r^n}{(b-r)^2r^n}.
\end{align}
To establish \eqref{eq:toshow2}, fix $n_0$ such that for all $n\geq n_0, r^n \geq \beta$. To begin with, observe that starting from $(0,i,0) \in S_n$, we have $\Tilde X_1 \sim \mathcal{P}(bi+\lambda)$. Subsequently, according to Bienaymé-Chebychev inequality :
$$\proba{\left| \mathcal{P}(bi+\lambda)-(bi+\lambda) \right| > (b-r)i} \leq \dfrac{bi+\lambda}{(b-r)^2i^2} \leq \dfrac{b+\lambda/r^n}{(b-r)^2r^n}.$$
This means that with high probability, 
$$\Tilde X_1 \in [bi+\lambda \pm (b-r)i].$$
Furthermore, knowing $\Tilde X_1$ we have : 
$$\Tilde X_2 \sim \mathcal{P}\left((a\Tilde X_1 + ci + \lambda)_+\right).$$ 
However, since $i \geq \beta$, with high probability :
$$a\Tilde X_1 + ci + \lambda \leq a(bi+\lambda + (b-r)i) + ci + \lambda = (ab+c+\varepsilon) i + a\lambda + a \leq 0,$$
and : 
$$\Tilde X_1 \geq bi+\lambda-(b-r)i = ri + \lambda > ri \geq r^{n+1}.$$
We deduce that
$$1 - \P_{(0,i,0)}(\Tilde X_1 \geq r^{n+1}, \Tilde X_2 = 0) \leq \dfrac{b+\lambda/r^n}{(b-r)^2r^n}.$$
Now, \eqref{eq:toshow2} implies,
\[
\forall x\in S_n: \P_x(X_3\in S_{n+1}) \ge p_n \coloneqq \left(1 - \dfrac{b+\lambda/r^n}{(b-r)^2r^n}\right)_+,
\]
and
\[
\sum_{n\ge 1} (1-p_n) \le \sum_{n\ge1} \dfrac{b+\lambda/r^n}{(b-r)^2r^n} < \infty.
\]
This establishes that the third assumption of Lemma~\ref{lem:strategy_transience} is satisfied. Consequently, the lemma implies that the Markov chain is transient, which concludes the proof.
\end{proof}

\appendix

\section[\texorpdfstring{Proof of transience when $a_1, \dots, a_p \geq 0$ and $a_1 + \dots + a_p > 1$.}{Proof of transience when a_1, ... , a_p >= 0 and a_1 + ... + a_p > 1.}]{Proof of transience when $a_1, \dots, a_p \geq 0$ and $a_1 + \dots + a_p > 1$.}
\label{app:Trans_p}

\begin{prop}
Let $p\geq 2$ and $a_1, \dots, a_p \geq 0$ such that $a_1 + \dots + a_p > 1$. Recall that for $n \geq 0$, $$\Tilde X_{n+1} \sim \mathcal{P}(a_1 \Tilde X_n + a_2 \Tilde X_{n-1} + \dots + a_p \Tilde X_{n-p+1}).$$
Then, the Markov chain $X_n := (\Tilde X_n, \dots, \Tilde X_{n-p+1})$ is transient.
\label{prop_Trans_p}
\end{prop}

\begin{proof}
Let us define the polynomial 
$$P(X) = X^p - a_1X^{p-1} - \dots a_{p-1}X - a_p.$$
Since $\displaystyle \lim_{x\to +\infty} P(x) = +\infty$, and $P(1) = 1 - a_1 - \dots - a_p < 0$, there exist $\theta > 1$ and $r\in (1,\theta)$ such that $P(\theta) = 0$, and
\begin{equation}
r^p - a_1 r^{p-1} - \dots - a_{p-1}r-a_p < 0.
\label{property_r}
\end{equation}
For the rest of the proof, fix $r$ such that \eqref{property_r} holds.

For $n \geq p$, let us define $m_n = n$ and the sets 
$$S_n = \left\{ (i_1, \dots, i_p) \in \N^p ~|~ i_1 \geq r^n, i_2 \geq r^{n-1}, \dots, i_p \geq r^{n-p+1} \right\}.$$
Note that since $a_1, \dots, a_p \geq 0$, assumption 1 and 2 are automatically satisfied. 

For assumption 3, let $n \geq p$ and $\mathcal{I} := (i_1, \dots, i_p) \in S_n$. Then, 
\begin{align*}
\P_\mathcal{I}(\Tilde X_1 < r^{n+1}) &= \proba{\mathcal{P}(a_1 i_1 + \dots a_p i_p + \lambda) < r^{n+1}} \\
&\leq \proba{\mathcal{P}(a_1 r^n + \dots a_p r^{n-p+1}) < r^{n+1}} \\
&= \mathbb{P}\Big[\mathcal{P}(a_1 r^n + \dots a_p r^{n-p+1})-(a_1 r^n + \dots a_p r^{n-p+1})\\
&\qquad\qquad < r^{n-p+1}(r^p - a_1 r^{p-1} - \dots - a_{p-1}r-a_p)\Big].
\end{align*}
From \eqref{property_r} and Bienaymé-Chebychev inequality, we deduce that : 
\begin{align*}
\P_\mathcal{I}(\Tilde X_1 < r^{n+1}) 
&\leq \mathbb{P}\Bigg[\Big| \mathcal{P}(a_1 r^n + \dots a_p r^{n-p+1})-(a_1 r^n + \dots a_p r^{n-p+1}) \Big| 
\\
&\qquad\qquad>r^{n-p+1}|r^p - a_1 r^{p-1} - \dots - a_{p-1}r-a_p|\Bigg] \\
&\leq \dfrac{(a_1 + \dots + a_p)r^{2(p-1)}}{r^n(r^p - a_1 r^{p-1} - \dots - a_{p-1}r-a_p)^2}
\end{align*}
Hence, assumption 3 of Lemma \ref{lem:strategy_transience} is satisfied by considering a sequence $(p_n)$ defined as :
$$p_n := \left( 1 - \dfrac{(a_1 + \dots + a_p)r^{2(p-1)}}{r^n(r^p - a_1 r^{p-1} - \dots - a_{p-1}r-a_p)^2} \right)_+.$$
This concludes the proof of Proposition \ref{prop_Trans_p}.
\end{proof}

\section{Proof of technical lemmas}
\label{app:technique}
%

\subsection{Proof of Lemma \ref{lem_discP}}
Coming back to the definition of $\text{Disc}(P)$ in \eqref{eq:disc},  $\text{Disc}(P) < 0$ 
writes
$$a^2b^2+4b^3-4a^3c-18abc-27c^2<0
$$
Let us look at the l.h.s. as a polynomial $S$ of degree $2$ in $c$
$$S(c)=-27c^2 -c(4a^3+18ab) +a^2b^2+4b^3$$
Note that its discriminant
\begin{align*}
    D_S&= (4a^3+18ab)^2 + 4\times 27 b^2(a^2+4b)\\
    &=16 (a^2 +3b)^3
\end{align*}
If $D_S<0$, since the dominant coefficient of $S$ is negative, $S$ is negative for all $c$.

Otherwise, $S$ admits two real roots which are exactly $c_\pm$ given by 
$$c_\pm= \frac{1}{27}(-2a^3-9ab\pm 2 (a^2+3b)^{3/2}).$$

The last condition arises simply by considering that if $S(0)= b^2(a^2+4b)>0$, then the two eigenvalues have different signs.

\subsection{Proof of Lemma \ref{lem:R(alpha_Q)}}
Recall that $\alpha_Q$ is the unique real root of $Q$ and that since $Q(0)=c<0$, $\alpha_Q$ is positive.

We now prove that $R(\alpha_Q) > 0$ where $R$ is given in \eqref{def:R}. 
Note that $R$ is a polynomial in $\alpha$, of discriminant $16(a^2+3b)$.\\
\noindent\underline{If $a^2+3b<0$}, then $R$ is positive for all values of $\alpha$, in particular $R(\alpha_Q) > 0$. \\
\noindent\underline{If $a^2+3b\ge0$}, the roots of $R$ are 
$$\alpha_{\pm} = -\dfrac{a}{3} \pm \dfrac{2}{3} \sqrt{a^2+3b}.$$
Note that $\alpha_-$ is always negative while $\alpha_+\in[-a/3,a/3]$.
\begin{itemize}
    \item In the case where $\alpha_+<0$, then $R(\alpha)>0$ for all positive $\alpha$, and as a consequence $R(\alpha_Q)>0$.
    \item
Let us now assume $\alpha_+>0$. In particular, this implies that $R(0)<0$ which writes $a^2+4b>0$.\\
Computing carefully, we can derive that 
$$\alpha_+^2 = \frac{-a^2}{3}+\frac{4}{3}(a^2+3b) -\frac{4a}{9}\sqrt{a^2+3b}\,,$$
\begin{align*}
    \alpha_+^3 = \frac{8}{27}(a^2+3b)^{3/2} +\frac{2a^2}{9}\sqrt{a^2+3b} -\frac{a^3}{27} -\frac{4}{9}(a^2+3b)\,,
\end{align*}
and finally
\begin{align*}
Q(\alpha_+)&= \alpha_+^3+a\alpha_+^2-b\alpha_++c\\
&= \dfrac{1}{27}\left( 2a^3+9ab+27c+2(a^2+3b)^{\frac{3}{2}} \right)\,.
\end{align*}
In the following, we prove that $Q(\alpha_+)<0$. Since $Q$ has a unique real root $\alpha_Q$, we will deduce that $\alpha_Q>\alpha_+$, and as a consequence $R(\alpha_Q)>0$. \\
In order to prove that $Q(\alpha_+)<0$, it is sufficient to prove that under our assumptions on $a$, $b$ and $c$
\begin{equation}
    \label{eq:c_1}
    c<\frac{-1}{27} (2a^3+9ab +2(a^2+3b)^{\frac{3}{2}})
\end{equation}
This derives directly from Lemma \ref{lem_discP} since we assumed $\Disc(P)<0, a^2+4b\ge 0$ and $c<0$, 
\end{itemize}

\subsection{Proof of Lemma \ref{Lemme ab+c}}

\begin{lemma}
\label{Lemme ab+c}
Let $b\in \R$, $a\ge0$ and $c<0$ such that $\Disc(P) < 0$. Then, $ab + c < 0$. 
\end{lemma}
\begin{proof}[Proof of Lemma \ref{Lemme ab+c}]
Let $a,b,c$ as in the statement of the lemma.
\hspace{1cm}
\begin{itemize}
\item If $b\le0$, since $c<0$ and $a\ge0$ the result is obvious. 

\item If $b>0$, then from Lemma \ref{lem_discP}, $\Disc(P)<0$ is equivalent to :
$$c < \dfrac{1}{27}(-2a^3-9ab)-\dfrac{2}{27}(a^2+3b)^\frac{3}{2}$$
which yields :
$$ab+c < \dfrac{1}{27}(-2a^3+18ab)-\dfrac{2}{27}(a^2+3b)^\frac{3}{2} = \dfrac{2}{27}\left[-a^3+9ab- (a^2+3b)^\frac{3}{2}\right]$$

It remains to prove that the r.h.s. is negative which can be written as \begin{align*}
-a^3 + 9ab < (a^2+3b)^{\frac{3}{2}} &\Longleftrightarrow (9ab-a^3)^2 < (a^2+3b)^3 \\
&\Longleftrightarrow -27a^4b+54a^2b^2-27b^3<0 \\
&\Longleftrightarrow -27(a^2-b)^2b < 0
\end{align*}
the last inequality being obviously true when $b>0$.
\end{itemize}
\end{proof}

\section*{Acknowledgements}

M.C. has been supported by the Chair ”Modélisation Mathématique et Biodiversité” of Veolia Environnement-École Polytechnique-Muséum national d’Histoire naturelle-Fondation X and by ANR project HAPPY (ANR-23-CE40-0007) and DEEV (ANR-20-CE40-0011-01). P.M. acknowledges partial support from ANR grant ANR-20-CE92-0010-01 and from Institut Universitaire de France.
\bibliographystyle{plain}

\end{document}